\documentclass[10pt]{article}
\usepackage{graphicx}
\usepackage{mathptmx}
\usepackage[english]{babel}
\usepackage{amssymb,amsmath,amsthm}
\usepackage{enumitem}
\usepackage{epstopdf}
\usepackage[T1]{fontenc}

\newcommand*{\textlabel}[2]{%
  \edef\@currentlabel{#1}
  \phantomsection
  #1\label{#2}
}

\newcommand{\R}{\mathbb{R}}

\newcommand{\E}{\mathbb{E}}

\providecommand{\norm}[1]{\ensuremath{\lVert#1\rVert}}

\providecommand{\tnorm}[1]{\ensuremath{\lvert\lvert\lvert #1\rvert\rvert\rvert}} 

\newcommand{\dd}{\mathrm{d}}

\newtheorem{theorem}{Theorem}[section]
\newtheorem{lemma}{Lemma}[section]

\newtheorem{proposition}{Proposition}[section]
\newtheorem{remark}{Remark}[section]

\title{A fully discrete approximation of the one-dimensional stochastic heat equation}

\author{%
{\sc
Rikard Anton\thanks{Email: rikard.anton@umu.se},
David Cohen\thanks{Corresponding author. Email: david.cohen@umu.se}
} \\[2pt]
Department of Mathematics and Mathematical Statistics,\\ 
Ume{\aa} University, 90187~Ume{\aa}, Sweden\\[6pt]
{\sc and}\\[6pt]
{\sc Lluis Quer-Sardanyons}\thanks{Email: quer@mat.uab.cat}\\[2pt]
Department of Mathematics, \\
Universitat Aut\`onoma de Barcelona, 08193~Bellaterra, Catalonia
}

\begin{document}
       

\maketitle

\begin{abstract}
A fully discrete approximation of the one-dimensional stochastic heat equation driven by multiplicative space-time white noise is presented.
The standard finite difference approximation is used in space and a stochastic exponential method
is used for the temporal approximation. Observe that the proposed exponential scheme does not suffer from any kind 
of CFL-type step size restriction. When the drift term and the diffusion coefficient are assumed to be globally Lipschitz, 
this explicit time integrator allows for error bounds in $L^q(\Omega)$, 
for all $q\geq2$, improving some existing results in the literature. 
On top of this, we also prove almost sure convergence of the numerical scheme.
In the case of non-globally Lipschitz coefficients, we provide sufficient conditions under which 
the numerical solution converges in probability to the exact solution. 
Numerical experiments are presented to illustrate the theoretical results.\\[1.5ex]
\end{abstract}


\section{Introduction}\label{sect:intro}

We study an explicit full numerical discretization of the
one-dimensional stochastic heat equation
\begin{align}
&\frac{\partial}{\partial t}u(t,x)  = \frac{\partial^2}{\partial x^2} u(t,x) +
f(t,x,u(t,x))+\sigma(t,x,u(t,x))\frac{\partial^2}{\partial t\partial x}W(t,x) \quad \mathrm{in}\: \
(0,\infty)\times (0,1), \nonumber \\
&u(t,0) = u(t,1)=0 \quad \mathrm{for}\:\ t\in(0,\infty),\nonumber \\
&u(0,x)=u_0 \quad \mathrm{for}\:\ x\in[0,1], \label{heateq}
\end{align}
where $W$ is a Brownian sheet on $[0,\infty)\times[0,1]$ defined on
some probability space $(\Omega,\mathcal{F},\mathbb{P})$ satisfying the usual conditions, and 
$u_0$ is a continuous function in $[0,1]$ such that $u_0(0)=u_0(1)=0$.
Assumptions on the
coefficients $f$ and $\sigma$ will be
specified below. As far as the spatial discretization is concerned, we use a standard finite difference scheme, as in \cite{gyongy1}. 
In order to discretize \eqref{heateq} with respect to the time variable, we consider an exponential method similar to the
time integrators used in \cite{MR3033008,cqs14,MR3484400} for
stochastic wave equations or in \cite{ac16s,Cohen2017} for
stochastic Schr\"odinger equations.

\smallskip

Our main aim is to improve the temporal rate of convergence that has been obtained by Gy\"ongy in the reference \cite{gyongy2}. 
Indeed, in \cite{gyongy2}, the explicit as well as the semi-implicit 
Euler-Maruyama scheme have been applied for the time discretization of
problem \eqref{heateq}. When the functions $f$ and $\sigma$ are
globally Lipschitz continuous in the third variable, a temporal convergence order of $\frac18-$ 
in the $L^q(\Omega)$-norm, for all $q\geq 2$,  is obtained for these numerical schemes 
(see Theorem~3.1 in \cite{gyongy2} for a precise statement). Our first objective is to see if an explicit
exponential method can provide a higher rate of convergence. 
In the present work, we answer this question positively and obtain the temporal rate $\frac14-$ (see the first part of Theorem~\ref{fullLip} below).
We note that, as in \cite{gyongy2}, the latter estimate for the $L^q(\Omega)$-error holds for any fixed $t\in (0,T]$ and uniformly in the 
spatial variable, where $T>0$ is some fixed time horizon. 
On the other hand, we should also remark that, in \cite{gyongy2}, a rate of convergence $\frac14$ could be obtained only in the case where
the initial condition $u_0$ belongs to $C^3([0,1])$. Finally, as in \cite{gyongy2},
we also prove that the exponential
scheme converges almost surely to the solution of \eqref{heateq}, uniformly with respect to time and space variables (cf. 
Theorem~\ref{th:as}). 

\smallskip

Our second objective consists in refining the above-mentioned temporal rate of convergence in order to 
end up with a convergence order which is exactly $\frac14$ and 
with an estimate which is uniform both with respect to time and space variables. 
To this end, we assume that the initial condition $u_0$ belongs 
to some fractional Sobolev space (see \eqref{eq:22} for the precise definition). Indeed, as it can be deduced from the second part of Theorem~\ref{fullLip}
and well-known Sobolev embedding results, in order to have the rate $\frac14$, the hypothesis on $u_0$ implies that it
is $\delta$-H\"older continuous for all $\delta\in (0,\frac12)$.
Eventually, as in \cite{gyongy2}, we remove the globally
Lipschitz assumption on the coefficients $f$ and $\sigma$ in equation \eqref{heateq}, 
and we prove convergence in probability for the proposed
explicit exponential integrator (see Theorem~\ref{th:proba} below).

\smallskip

We should point out that there are also other important advantages with using the exponential method proposed here. Namely, first,  it does 
not suffer a step size restriction (imposed by a CFL condition) as the explicit Euler-Maruyama
scheme from \cite{gyongy2}. Secondly, it is an explicit scheme and
therefore has implementation advantages over the implicit
Euler-Maruyama scheme studied in \cite{gyongy2}. 
These facts will be illustrated numerically. 

\smallskip

The numerical analysis of the stochastic heat equation
\eqref{heateq} is an active research area. Without being too
exhaustive, beside the above mentioned papers \cite{gyongy1} and
\cite{gyongy2}, we mention the following works regarding numerical
discretizations of stochastic parabolic partial differential
equations: \cite{gyongy1,Yan,MR2916876,MR3290962} (spatial approximations);
\cite{MR1352735,MR1341554,gn97,anz98,MR1683281,m2an0068,DavieGaines,MR1951901,
MR1953619,ps05,Walsh,mm05,ritter,MR2471778,j09,MR2646103,MR2830608,j11,klns11,MR3047942,wg12,MR3027891,MR3101829,MR3320928,MR3534472,drw016}
(temporal and full discretizations); \cite{TaWangNie,lpt17}
(stability). Observe that most of these references are concerned
with an interpretation of stochastic partial differential equations
in Hilbert spaces and thus error estimates are provided in the
$L^2([0,1])$ norm (or similar norms). The reader is referred to the
monographs \cite{MR2856611,MR3154916,MR3308418} for a more
comprehensive reference list.

\smallskip

In the present publication, we follow a similar approach as in \cite{cqs14} and \cite{gyongy2}. 
The main idea consists in establishing suitable \textit{mild} forms for the spatial approximation $u^M$ and 
for the fully discretization scheme $u^{M,N}$. 
The obtained mild equations, together with some auxiliary results and taking into account the hypotheses on 
the coefficients and initial data, will allow us to deal with the $L^q(\Omega)$-error  
\[
\Bigl(\E[|u^M(t,x)-u^{M,N}(t,x)|^q]\Bigr)^{\frac 1{q}},
\]
for all $q\geq 2$. The $L^q(\Omega)$-error comparing $u^M$ with the exact solution of \eqref{heateq} has already been studied in 
\cite{gyongy1}.

\smallskip

The paper is organized as follows. In Section \ref{section:Lipschitz}, we study the numerical approximation of the solution
to equation \eqref{heateq} in the case of globally Lipschitz continuous coefficients. More precisely, 
we first recall the spatial discretization $u^M$ of \eqref{heateq} and prove some properties of $u^M$ needed in the sequel. 
Next, we introduce the full discretization scheme and prove that it satisfies a suitable mild form, and provide three auxiliary results which will be invoked in 
the convergence results' proofs. At this point, we state and prove the main result on $L^q(\Omega)$-convergence along with some 
numerical experiments illustrating its conclusion. Section \ref{section:Lipschitz} concludes with the result on almost sure convergence,
where we also provide some numerical experiments. Finally, Section \ref{section:nonLipschitz} is devoted to deal with the convergence 
in probability of the numerical solution to the exact solution of \eqref{heateq}, 
in the case where the coefficients $f$ and $\sigma$ are non-globally Lipschitz continuous.

\smallskip

Observe that, throughout this article, $C$ will denote a generic constant that may vary from line to line. 


\section{Error analysis for globally Lipschitz continuous coefficients}\label{section:Lipschitz}

This section is divided into three subsections. We begin by stating the assumptions
we will make and by recalling the mild solution of \eqref{heateq}. The first subsection is dedicated
to recalling the finite difference approximation from \cite{gyongy1} and some (new) results about it.
In the second subsection, we numerically integrate the resulting semi-discrete system of stochastic differential equations in time
to obtain a full approximation of \eqref{heateq}. We also state and prove our main result about convergence
in the $2p$-th mean. Finally, in the third subsection, we prove almost sure convergence of
the full approximation to the exact solution. In addition, numerical experiments are provided to illustrate the theoretical
results of this section.

In this section, we shall make the following assumptions on the coefficients
of the stochastic heat equation \eqref{heateq}: for a given positive real number $T$, there exist a constant $C$ such that
\begin{align}\label{L}
|f(t,x,u)-f(t,y,v)|+|\sigma(t,x,u)-\sigma(t,y,v)|\leq C\bigl(|x-y|+|u-v|\bigr),\tag{L}
\end{align}
for all $t\in[0,T]$, $x,y\in[0,1]$, $u,v\in\R$, and
\begin{align}\label{LG}
|f(t,x,u)|+|\sigma(t,x,u)|\leq C(1+|u|),\tag{LG}
\end{align}
for all $t\in[0,T]$, $x\in[0,1]$, $u\in\R$.
Assume also that the initial condition $u_0$ defines a continuous function on $[0,1]$ with 
$u_0(0)=u_0(1)=0$.
The assumptions \eqref{L} and \eqref{LG} imply existence and uniqueness
of a solution $u$ of equation \eqref{heateq} on the time interval $[0, T]$, see e.g. Theorem 3.2 and Exercise 3.4 in
\cite{walsh1}.
Let us recall that,  for a stochastic basis $(\Omega,\mathcal{F},(\mathcal{F}_t)_{t\geq 0}, \mathbb{P})$,
a solution to equation \eqref{heateq} is an $\mathcal{F}_t$-adapted continuous process
$\{u(t,x), (t,x)\in [0,T]\times[0,1]\}$ satisfying that, for every $\Phi\in C^\infty(\mathbb{R}^2)$
such that $\Phi(t,0)=\Phi(t,1)=0$ for all $t\geq 0$, we have
\begin{align}
 \int_0^1 u(t,x)\Phi(t,x)\, \dd x = & \int_0^1 u_0(x)\Phi(t,x)\, \dd x \nonumber \\
 & \quad + \int_0^t\int_0^1 u(s,x) \left(\frac{\partial^2 \Phi}{\partial x^2}(s,x) +
 \frac{\partial \Phi}{\partial s}(s,x)\right)\,  \dd x\,  \dd s \nonumber \\
 & \quad + \int_0^t\int_0^1 f(s,x,u(s,x)) \Phi(s,x)\,  \dd x\,  \dd s \nonumber \\
 & \quad + \int_0^t\int_0^1 \sigma(s,x,u(s,x)) \Phi(s,x)\,  W(\dd s,\dd x), \quad \mathbb{P}\text{-a.s.},
 \label{eq:13}
\end{align}
for all $t\in [0,T]$. It is well-known that the above equation implies the following {\it{mild}} form
for \eqref{heateq}:
\begin{equation}
\label{exactsol}
\begin{aligned}
u(t,x)&= \int_0^1 G(t,x,y)u_0(y)\,\dd y+\int_0^t\int_0^1 G(t-s,x,y)f(s,y,u(s,y))\,\dd y\,\dd s\\
&\quad+\int_0^t\int_0^1 G(t-s,x,y)\sigma(s,y,u(s,y))\,W(\dd s,\dd y), \quad \mathbb{P}\text{-a.s.},
\end{aligned}
\end{equation}
where $G(t,x,y)$ is the Green function of the linear heat equation with homogeneous Dirichlet boundary
conditions:
\[
 G(t,x,y)=\sum_{j=1}^\infty e^{-j^2 \pi^2 t} \varphi_j(x)\varphi_j(y), \quad t>0,\, x,y\in[0,1],
\]
with $\varphi_j(x):=\sqrt{2} \sin(j\pi x)$, $j\geq 1$. Note that these functions form an
orthonormal basis of $L^2([0,1])$.

\subsection{Spatial discretization of the stochastic heat equation}\label{sect:spatial}
In this subsection we recall the finite difference discretization
and some results obtained in \cite{gyongy1}. In addition to this, we
show new regularity results for the approximated Green function
$G^M(t,x,y)$ defined below, and for the space discrete
approximation, which will be needed in the sequel.

Let $M\geq 1$ be an integer and define the grid points $x_m=\frac{m}{M}$ for $m=0,\ldots,M$,
and the mesh size $\Delta x=\frac{1}{M}$. We now use the standard finite difference scheme
for the spatial approximation of \eqref{heateq} from \cite{gyongy1}. Let the process $u^M(t,\cdot)$ be defined
as the solution of the system of stochastic differential equations (for $m=1,\ldots,M-1$)
\begin{equation}
\begin{aligned}
\text du^M(t,x_m)&=M^2\left(u^M(t,x_{m+1})-2u^M(t,x_m)+u^M(t,x_{m-1})\right)\,\dd t\\
&\quad+f(t,x_m,u^M(t,x_m))\,\dd t\\
&\quad+M\sigma(t,x_m,u^M(t,x_m))\,\dd (W(t,x_{m+1})-W(t,x_m))
\end{aligned}
\label{fdeq}
\end{equation}
with Dirichlet boundary conditions
$$u^M(t,0)=u^M(t,1)=0,$$
and initial value
$$u^M(0,x_m)=u_0(x_m),$$
for $m=1,\ldots,M-1$. For $x\in[x_m,x_{m+1})$ we define
\begin{equation}
u^M(t,x):=u^M(t,x_m)+(Mx-m)(u^M(t,x_{m+1})-u^M(t,x_m)).
\label{eq:-1}
\end{equation}
We use the notations $u_m^M(t):=u^M(t,x_m)$ and $W_m^M(t):=\sqrt{M}(W(t,x_{m+1})-W(t,x_m))$, for $m=1,\ldots,M-1$
and write the system \eqref{fdeq} as
\begin{align*}
\text du^M_m(t)&=M^2\sum_{i=1}^{M-1}D_{mi}u^M_i(t)\,\dd t+f(t,x_m,u^M_m(t))\,\dd t\\
&\quad+\sqrt{M}\sigma(t,x_m,u^M_m(t))\,\dd W_m^M(t),
\end{align*}
with initial value
$$u_m^M(0)=u_0(x_m),$$
for $m=1,\ldots,M-1,$ where $D=(D_{mi})_{m,i}$ is a square matrix of size $M-1$, with elements $D_{mm}=-2$, $D_{mi}=1$ for $|m-i|=1$, $D_{mi}=0$ for $|m-i|>1$.
Also $W^M(t):=(W^M_m(t))_{m=1}^{M-1}$ is an $M-1$ dimensional Wiener process. Observe that the matrix $M^2D $ has eigenvalues
$$\lambda_j^M:=-4\sin^2\left(\frac{j\pi}{2M}\right)M^2=-j^2\pi^2c_j^M,$$
where
$$\frac{4}{\pi^2}\leq c_j^M:=\frac{\sin^2\left(\frac{j\pi}{2M}\right)}{\left(\frac{j\pi}{2M}\right)^2}\leq 1,$$
for $j=1,2,\ldots,M-1$ and every $M\geq 1$.

Using the variation of constants formula, the exact solution to \eqref{fdeq} reads
\begin{equation}
\begin{aligned}
u^M(t,x_m)&=\frac{1}{M}\sum_{l=1}^{M-1}\sum_{j=1}^{M-1}\exp(\lambda_j^M t)\varphi_j(x_m)\varphi_j(x_l)u_0(x_l)\\
&\quad+\int_0^t\frac{1}{M}\sum_{l=1}^{M-1}\sum_{j=1}^{M-1}\exp(\lambda_j^M(t-s))\varphi_j(x_m)\varphi_j(x_l)f(s,x_l,u^M(s,x_l))\,\dd s\\
&\quad+\int_0^t\frac{1}{\sqrt{M}}\sum_{l=1}^{M-1}\sum_{j=1}^{M-1}\exp(\lambda_j^M(t-s))\varphi_j(x_m)\varphi_j(x_l)\sigma(s,x_l,u^M(s,x_l))\,\dd W_l^M(s),
\end{aligned}
\label{spacediscsol}
\end{equation}
where we recall that $\varphi_j(x):=\sqrt{2}\sin(jx\pi)$ for $j=1,\ldots,M-1$.

We next define the discrete kernel $G^M(t,x,y)$ by
\begin{equation}
G^M(t,x,y):=\sum_{j=1}^{M-1}\exp(\lambda_j^M t)\varphi_j^M(x)\varphi_j(\kappa_M(y)),
\label{eq:1}
\end{equation}
where $\kappa_M(y):=\frac{[My]}{M}$, $\varphi_j^M(x):=\varphi_j(x_l)$ for $x=x_l$ and
\[
\varphi_j^M(x):=\varphi_j(x_l)+(Mx-l)\left(\varphi_j(x_{l+1})-\varphi_j(x_l)\right),
\quad \text{if } x\in(x_l,x_{l+1}].
\]

With these definitions in hand, one sees that the semi-discrete solution $u^M$ satisfies the mild equation:
\begin{equation}
\begin{aligned}
u^M(t,x)&=\int_0^1 G^M(t,x,y)u_0(\kappa_M(y))\,\dd y\\
&\quad+\int_0^t\int_0^1 G^M(t-s,x,y)f(s,\kappa_M(y),u^M(s,\kappa_M(y)))\,\dd y\,\dd s\\
&\quad+\int_0^t\int_0^1 G^M(t-s,x,y)\sigma(s,\kappa_M(y),u^M(s,\kappa_M(y)))\,\dd W(s,y)
\end{aligned}
\label{spaceapp}
\end{equation}
$\mathbb{P}$-a.s., for all $t\geq 0$ and $x\in[0,1].$

Next, we proceed by collecting some useful results for the error analysis of the fully discrete numerical discretization
presented in the next subsection. The following two results are proved in \cite{gyongy1}.
Recall that $u^M$ is the space discrete approximation given by \eqref{spaceapp} and that $u$ is the exact solution given by equation \eqref{exactsol}.

\begin{proposition}[Proposition $3.5$ in \cite{gyongy1}]\label{xreg}
Assume that $u_0\in C([0,1])$ with $u_0(0)=u_0(1)=0$,
and that the functions $f$ and $\sigma$ satisfy the condition \eqref{LG}.
Then, for every $p\geq 1$, there exists a constant $C$ such that
$$
\sup_{M\geq1}\sup_{(t,x)\in[0,T]\times[0,1]} \E[|u^M(t,x)|^{2p}]\leq C.
$$
\end{proposition}

\begin{theorem}[Theorem $3.1$ in \cite{gyongy1}]\label{th:space}
Assume that $f$ and $\sigma$ satisfy the conditions
\eqref{L} and \eqref{LG}, and that $u_0\in C([0,1])$
with $u_0(0)=u_0(1)=0$. Then, for every $0<\alpha<\frac14$, $p\geq 1$ and for every $t>0$,
there is a constant $C=C(\alpha,p,t)$ such that
\begin{align}\label{xconv}
\sup_{x\in[0,1]}\Big(\E[|u^M(t,x)-u(t,x)|^{2p}]\Big)^{\frac{1}{2p}}\leq C(\Delta x)^{\alpha}.
\end{align}
We recall that $\Delta x=1/M$ is the mesh size in space.
Moreover, $u^M(t,x)$ converges to $u(t,x)$ almost surely as $M\rightarrow\infty$,
uniformly in $t\in[0,T]$ and $x\in [0,1]$, for every $T>0$.

If $u_0$ is sufficiently smooth (e.g. $u_0\in C^3([0,1])$) then for every $T>0$, estimate \eqref{xconv}
holds with $\alpha=\frac12$ and with the same constant $C$ for all $t\in[0,T]$ and integer $M\geq 1$.
\end{theorem}

\medskip

We will also make use of the following estimates on the discrete Green function.

\smallskip

\begin{lemma}\label{greenreg}
There is a constant $C$ such that the following estimates hold:
\begin{enumerate}[label=(\roman*)]
\item For all $0<s<t\leq T$:
\begin{align}\label{greenreg1}
\sup_{M\geq 1} \sup_{x\in[0,1]}\int_0^s\int_0^1|G^M(t-r,x,y)-G^M(s-r,x,y)|^2\,\dd y\,\dd r\leq C(t-s)^{1/2}.
\end{align}
\item For all $t\in (0,T]$:
$$
\sup_{M\geq1}\sup_{x\in[0,1]}\int_0^1|G^M(t,x,y)|^2\,\dd y\leq C\frac{1}{\sqrt{t}}.
$$
\item For all $0<s<t\leq T$ and $\alpha\in (\frac12,\frac52)$:
$$
\sup_{M\geq1} \sup_{x\in[0,1]}\int_0^1 |G^M(t,x,y)-G^M(s,x,y)|^2\,\dd y\leq Cs^{-\alpha}(t-s)^{\alpha-\frac12}.
$$
\end{enumerate}
\end{lemma}
\begin{proof}
Recall that
$$G^M(t,x,y)=\sum_{j=1}^{M-1}\exp(\lambda_j^M t)\varphi_j^M(x)\varphi_j(\kappa_M(y)),$$
where $\kappa_M(y)=\frac{[My]}{M},$ $\varphi_j^M(x)=\varphi_j\left(\frac{l}{M}\right)$ for $x=\frac{l}{M}$
and
\[
\varphi_j^M(x)=\varphi_j\left(\frac{l}{M}\right)+(Mx-l)\left(\varphi_j
\left(\frac{l+1}{M}\right)-\varphi_j\left(\frac{l}{M}\right)\right),\quad \text{if }
x\in \left(\frac{l}{M},\frac{l+1}{M}\right].
\]

We first prove $(i)$. Observe that a general version of this result is used in the proof of
\cite[Lem.~3.6]{gyongy1} (see the term $A_1^{2p}$ therein).
Using the definition of the discrete Green function, we have
\begin{align*}
&\int_0^s \int_0^1 |G^M(t-r,x,y)-G^M(s-r,x,y)|^2\,\dd y\,\dd r  \\
&\quad= \int_0^s\int_0^1\left|\sum_{j=1}^{M-1}(\exp(\lambda_j^M (t-r))-\exp(\lambda_j^M (s-r)))\varphi_j^M(x)\varphi_j(\kappa_M(y))\right|^2\,\dd y\,\dd r.
\end{align*}
At this point, we use the fact that  the vectors
\[
 e_j=\left(\sqrt{\frac2M} \sin\left(j\frac kM \pi\right),\; k=1,\dots,M-1\right), \quad j=1,\dots,M-1,
\]
form an orthonormal basis of $\mathbb{R}^{M-1}$, which implies that
\begin{equation}
\int_0^1 \varphi_j(\kappa_M(y)) \varphi_l(\kappa_M(y))\, \dd y= \delta_{\{j=l\}}.
\label{eq:8}
\end{equation}
Hence, using also the definitions of $\varphi_j^M$ and $\lambda_j^M$,
\begin{align*}
&\int_0^s \int_0^1 |G^M(t-r,x,y)-G^M(s-r,x,y)|^2\,\dd y\,\dd r\\
&\quad=\int_0^s\sum_{j=1}^{M-1} |\exp(\lambda_j^M (t-r)-\exp(\lambda_j^M (s-r))|^2|\varphi_j^M(x)|^2\,\dd r\\
&\quad\leq C\sum_{j=1}^{M-1}\int_0^s\exp(\lambda_j^M (s-r))^2\,\dd r |1-\exp(\lambda_j^M(t-s))|^2 \\
&\quad\leq C\sum_{j=1}^{M-1}\int_0^s\exp(-2j^2\pi^2 c_j^M (s-r))\,\dd r (1-\exp(-j^2\pi^2c_j^M(t-s)))^2 \\
&\quad\leq C\sum_{j=1}^\infty j^{-2}(j^4(t-s)^2\wedge 1).
\end{align*}
Here we have used that $1-\exp(-x)\leq x$, and that $(c_j^M)^{-1}$ is bounded. Let
$N:=\left[\frac{1}{\sqrt{t-s}}\right]$, where $[\cdot]$ denotes the integer part, and observe that 
(by comparing sums with integrals) 
\begin{align*}
\sum_{j=1}^\infty j^{-2}(j^4(t-s)^2\wedge 1) & =
\sum_{j=1}^N j^2 (t-s)^2 + \sum_{j=N+1}^\infty j^{-2} \\
& \leq C (t-s)^2 (N+1)^3 + (N+1)^{-1} \\
& \leq C (t-s)^2 \left(\left[\frac{\sqrt{t-s}+1}{\sqrt{t-s}}\right]\right)^3 + (N+1)^{-1} \\
& \leq C (t-s)^2 \left(\left[\frac{1}{\sqrt{t-s}}\right]\right)^3 +
\left(\frac{1}{\sqrt{t-s}}\right)^{-1}\\
& \leq C (t-s)^{\frac12}.
\end{align*}
This proves part $(i)$.
The proof of $(ii)$ follows by similar arguments as those used in the proofs of
 \cite[Lem.~8.1, Thm~8.2]{Walsh}. First note that, as above, we have
\begin{align*}
\int_0^1 |G^M(t,x,y)|^2\,\dd y&=\int_0^1 \left|\sum_{j=1}^{M-1}\exp(\lambda_j^M t)\varphi_j^M(x)\varphi_j(\kappa_M(y))\right|^2\,\dd y\\
&\leq C\sum_{j=1}^{M-1}\exp(-2j^2\pi^2c_j^M t).
\end{align*}
The estimate in $(ii)$ now follows from the inequality
$$\sum_{j=1}^{M-1}\exp(-2j^2\pi^2c_j^M t)\leq C\left(M\wedge\frac{1}{\sqrt{2c_j^M}\pi\sqrt{t}}\right),$$
which is proved in \cite[Lem.~8.1]{Walsh}.

We now prove $(iii)$. Using the definition of the discrete Green function, properties of $\varphi_j$, and the definition of $\lambda_j^M$,
we have
\begin{align*}
\int_0^1|G^M(t,x,y)-G^M(s,x,y)|^2\,\dd y &\leq \sum_{j=1}^{M-1} |\exp(\lambda_j^M t)-\exp(\lambda_j^M s)|^2\\
&\leq \sum_{j=1}^{M-1} |\exp(-j^2\pi^2 c_j^M s)|^2|1-\exp(-j^2\pi^2 c_j^M(t-s))|^2.
\end{align*}
Since $1-\exp(-x)\leq x$ and $\exp(-x^2)\leq C_{\alpha}|x|^{-\alpha}$, for all $\alpha\in \mathbb{R}$, it follows that
\begin{align*}
\int_0^1|G^M(t,x,y)-G^M(s,x,y)|^2\,\dd y &\leq C_{\alpha}\sum_{j=1}^{M-1} j^{-2\alpha} s^{-\alpha}(1\wedge j^4(t-s)^2)\\
&\leq\tilde C_1(t-s)^2s^{-\alpha}\sum_{j=1}^N j^{4-2\alpha}+\tilde C_2s^{-\alpha}\sum_{j=N+1}^\infty j^{-2\alpha},
\end{align*}
where $N=\left[\frac{1}{\sqrt{t-s}}\right]$
and $\tilde C_1$ and $\tilde C_2$ are independent of $t$ and $s$.
We now estimate these two terms as we did in the proof of part $(i)$.
Namely, whenever $\alpha<\frac52$ we have that
\begin{align*}
(t-s)^2s^{-\alpha}\sum_{j=1}^N j^{4-2\alpha}&\leq C(t-s)^2s^{-\alpha}(N+1)^{5-2\alpha}\\
&\leq C(t-s)^{\alpha-1/2}s^{-\alpha},
\end{align*}
using the fact that $N+1\leq\frac{1+\sqrt{t-s}}{\sqrt{t-s}}\leq\frac{C_T}{\sqrt{t-s}}$.
For the second term, if $\alpha>\frac12$ we obtain
\begin{align*}
s^{-\alpha}\sum_{j=N+1}^\infty j^{-2\alpha}&=s^{-\alpha}(N+1)^{-2\alpha}+s^{-\alpha}\sum_{j=N+2}^\infty j^{-2\alpha}
\leq Cs^{-\alpha}(N+1)^{1-2\alpha}\\
&\leq (t-s)^{\alpha-1/2}s^{-\alpha}.
\end{align*}
Collecting these two estimates leads to the conclusion of the theorem.
\end{proof}

\medskip

For the numerical analysis of the exponential method applied to the nonlinear stochastic heat equation \eqref{heateq} presented in the next subsection,
the initial data $u_0$ will be in the space $H^\alpha([0,1])$, which we now define.
For $\alpha\in\mathbb{R}$, we define the space $H^\alpha([0,1])$
to be the set of functions $g\colon[0,1]\to\R$ such that
\begin{equation}
\norm{g}_{\alpha}=\left(\sum_{j=1}^\infty(1+j^2)^\alpha|\left\langle g,\varphi_j\right\rangle|^2\right)^{1/2}<\infty,
\label{eq:22}
\end{equation}
where we recall that $\varphi_j(x)=\sqrt{2}\sin(jx\pi)$, for $j\geq 1$.
The inner product in the above sum stands for the usual $L^2([0,1])$ inner product.
Further restrictions on $\alpha$ will be made in the results below.
For the sake of simplicity, the space $H^\alpha([0,1])$ will be denoted by
$H^\alpha$.
Note that this space is a
subspace of the fractional Sobolev space of fractional order $\alpha$ and integrability order
$p=2$ (see \cite{Triebel}). Moreover, for any $\alpha>\frac12$, the space
$H^\alpha$ is continuously embedded in the space of $\delta$-H\"older-continuous functions
for all $\delta\in (0,\alpha-\frac12)$ (see, e.g., \cite[Thm.~8.2]{MR2944369}).

\medskip

Finally, we need the following regularity results for the finite difference approximation $u^M$ given by \eqref{spaceapp}.
\begin{proposition}\label{reg}
Assume that $f$ and $\sigma$ satisfy the condition \eqref{LG}. 
\begin{enumerate}
\item Assume that $u_0\in C([0,1])$ with $u_0(0)=u_0(1)=0$.
For any $0< s\leq t\leq T$, any $p\geq 1$, and $\frac12<\alpha<\frac52$, we have
$$
\sup_{M\geq 1} \sup_{x\in[0,1]}\E[|u^M(t,x)-u^M(s,x)|^{2p}]\leq Cs^{-\alpha p}(t-s)^{\nu p},
$$
where $\nu=\frac12 \land (\alpha-\frac12)$.
\item Assume that $u_0\in H^{\beta}([0,1])$, 
with $u_0(0)=u_0(1)=0$, for some $\beta>\frac{1}{2}$. For any $0\leq s\leq t\leq T$ and any $p\geq 1$, we have
$$
\sup_{M\geq 1} \sup_{x\in[0,1]}\E[|u^M(t,x)-u^M(s,x)|^{2p}]\leq C(t-s)^{\tau p},
$$
where $\tau = \frac12\wedge(\beta-\frac12)$.
\end{enumerate}
\end{proposition}
\begin{proof}
For ease of presentation, we consider functions $f(u)$ and $\sigma(u)$ depending only on $u$.
Let us first define
\begin{align*}
F^M(t,x)&:=\int_0^t\int_0^1 G^M(t-s,x,y)f(u^M(s,y))\,\dd y\,\dd s\\
H^M(t,x)&:=\int_0^t\int_0^1 G^M(t-s,x,y)\sigma(u^M(s,y))\,\dd W(s,y).
\end{align*}
Then we have
\begin{align*}
u^M(t,x)-u^M(s,x)&=\int_0^1(G^M(t,x,y)-G^M(s,x,y))u_0(\kappa_M(y))\,\dd y\\
&\quad+F^M(t,x)-F^M(s,x)\\
&\quad+H^M(t,x)-H^M(s,x).
\end{align*}
By \cite[Lem.~3.6]{gyongy1}, the last two terms can be estimated by
\begin{align}\label{est1}
\E[|F^M(t,x)-F^M(s,x)|^{2p}]+\E[|H^M(t,x)-H^M(s,x)|^{2p}]\leq C|t-s|^\frac{p}{2}.
\end{align}
It remains to estimate the term involving $u_0$.

Assume first that $u_0\in C([0,1])$. We use the third part of Lemma~\ref{greenreg} to get the following estimate:
\begin{align*}
&\left(\E\left[\left|\int_0^1(G^M(t,x,y)-G^M(s,x,y))u_0(\kappa_M(y))\,\dd y\right|^{2p}\right]\right)^{1/p}\\
&\quad=\left|\int_0^1(G^M(t,x,y)-G^M(s,x,y))u_0(\kappa_M(y))\,\dd y\right|^2\\
&\quad\leq C\int_0^1 |G^M(t,x,y)-G^M(s,x,y)|^2|u_0(\kappa_M(y))|^2\,\dd y\\
&\quad\leq Cs^{-\alpha}(t-s)^{\alpha-\frac{1}{2}}.
\end{align*}
Collecting the above estimates and taking into account that $s^{-\alpha p}\geq T^{-\alpha p}$ in \eqref{est1}, we get
$$
\sup_{M\geq 1} \sup_{x\in[0,1]}\E[|u^M(t,x)-u^M(s,x)|^{2p}]\leq
C s^{-\alpha p}(t-s)^{\nu p},
$$
where $\nu=\frac12 \land (\alpha-\frac12)$.

Assume now that $u_0\in H^{\beta}([0,1])$ for some $\beta>\frac{1}{2}$.
Using the explicit expression of $G^M$, Cauchy-Schwarz inequality and that $1-\exp(-x)\leq x$,
we have
\begin{align*}
&\left|\int_0^1(G^M(t,x,y)-G^M(s,x,y))u_0(\kappa_M(y))\,\dd y\right|^{2p}\\
&\quad=\left|\sum_{j=1}^{M-1}(\exp(\lambda_j^M t)-\exp(\lambda_j^M s))\langle u_0(\kappa_M(y)),\varphi_j(\kappa_M(y))\rangle\varphi_j^M(x)\right|^{2p}\\
&\quad\leq \left(\sum_{j=1}^{M-1} |\exp(\lambda_j^M t)-\exp(\lambda_j^M s)| |\langle u_0,\varphi_j\rangle|\right)^{2p}\\
&\quad \leq C\left(\sum_{j=1}^{M-1} j^{-2\beta}|\exp(\lambda_j^M t)-\exp(\lambda_j^M s)|^2\right)^p\left(\sum_{j=1}^\infty j^{2\beta}|\langle u_0,\varphi_j\rangle|^2\right)^p\\
&\quad\leq C\left(\sum_{j=1}^{M-1}j^{-2\beta}\exp(2\lambda_j^M s)|\exp(\lambda_j^M (t-s))-1|^2\right)^p\norm{u_0}_{\beta}^{2p}\\
&\quad\leq C\left(\sum_{j=1}^{\infty}j^{-2\beta}(j^4(t-s)^2\wedge 1)\right)^p.
\end{align*}
Here we have used that $\langle u_0(\kappa_M(y)),\varphi_j(\kappa_M(y))\rangle = \langle u_0, \varphi_j\rangle$,
which can be verified by a simple calculation (see equation $(21)$ in \cite{Quer-Sardanyons2006}).
Furthermore, for $\beta>\frac{5}{2}$, we have
$$\sum_{j=1}^{\infty}j^{-2\beta}(j^4(t-s)^2\wedge 1)\leq C(t-s)^2.$$
On the other hand, if $\beta\in (\frac{1}{2},\frac{5}{2}]$,
$$\sum_{j=1}^{\infty}j^{-2\beta}(j^4(t-s)^2\wedge 1) = (t-s)^2 \sum_{j=1}^{N}j^{4-2\beta} +
\sum_{j=N+1}^\infty j^{-2\beta},$$
where $N=\left[\frac{1}{\sqrt{t-s}}\right]$, and $[\cdot]$ denotes the integer part.
Note that
$$(t-s)^2\sum_{j=1}^{N}j^{4-2\beta}\leq C(t-s)^{\beta-\frac{1}{2}}$$
and
$$\sum_{j=N+1}^\infty j^{-2\beta}\leq C(t-s)^{\beta-\frac{1}{2}}.$$
Hence, we arrive at the estimate
\begin{align}\label{est2}
\E\left[\left|\int_0^1(G^M(t,x,y)-G^M(s,x,y))u_0(\kappa_M(y))\,\dd y\right|^{2p}\right]\leq C(t-s)^{\gamma p},
\end{align}
where $\gamma = 2 \wedge (\beta-\frac{1}{2})$, for $\beta>\frac{1}{2}$. By the estimates \eqref{est1} and \eqref{est2}
we have
\begin{align*}
\sup_{M\geq 1} \sup_{x\in[0,1]}\E[|u^M(t,x)-u^M(s,x)|^{2p}]&\leq C(|t-s|^\frac{p}{2}+|t-s|^{\gamma p})\\
&\leq C|t-s|^{\tau p},
\end{align*}
where $\tau=\frac12 \wedge (\beta-\frac12)$, for $\beta>\frac{1}{2}$.
\end{proof}

\subsection{Full discretization: $\boldmath{L^{2p}(\Omega)}$-convergence}
\label{sect:temporal}

This section is devoted to introduce the time discretization of
the semi-discrete problem presented in the previous subsection, which will be
denoted by $u^{M,N}$. Next we prove properties of $u^{M,N}$
which will be needed in the sequel and we will state and prove the
main result of the present section (cf. Theorem \ref{th:time}
below). Finally, some numerical experiments will be performed in
order to illustrate the theoretical results obtained so far.

\medskip

We start by discretizing the space discrete solution \eqref{spacediscsol} in time using an exponential integrator.
For an integer $N\geq1$ and some fixed final time $T>0$, let $\Delta t=\frac{T}{N}$ and define the discrete times $t_n=n\Delta t$ for $n=0,1,\ldots,N$.
For simplicity of presentation, we consider that the functions $f$ and $\sigma$ only depend
on the third variable.
Let us now consider the mild equation \eqref{spacediscsol} on the small time interval $[t_n,t_{n+1}]$ written
in a more compact form
(recall the notation $u^M_m(t)=u^M(t,x_m)$), as follows:
$$
u^M(t_{n+1})=e^{A\Delta t}u^M(t_n)+\int_{t_n}^{t_{n+1}}e^{A(t_{n+1}-s)}F(u^M(s))\,\dd s+
\int_{t_n}^{t_{n+1}}e^{A(t_{n+1}-s)}\Sigma(u^M(s))\, \dd W^M(s),
$$
with the finite difference matrix $A:=M^2D$, the vector $F(u^M(s))$
with entries $f(u^M_m(s))$ for $m=1, 2, \ldots, M-1$, and the
diagonal matrix $\Sigma(u^M(s))$ with elements
$\sqrt{M}\sigma(u^M_m(s))$ for $m=1, 2, \ldots, M-1$. The matrix $D$
has been defined in Section \ref{sect:spatial}. We next discretize
the integrals in the above mild equation by freezing the integrands
at the left endpoints of the intervals, so we obtain the explicit
exponential integrator (omitting the explicit dependence on $M$ for
clarity)
\begin{equation}
\begin{aligned}
{{\mathcal U}}^0&:=u^M(0),\\
{{\mathcal U}}^{n+1}&:=e^{A\Delta t}\bigl({{\mathcal U}}^n+F({{\mathcal U}}^n)\Delta t+\Sigma({{\mathcal U}}^n)\Delta W^n\bigr),
\end{aligned}
\label{sexp}
\end{equation}
where the terms $\Delta W^n:=W^{M}(t_{n+1})-W^M(t_n)$ denote the $(M-1)$-dimensional Wiener increments.
The above formulation of the exponential integrator will be used for the practical computations presented below.

\begin{remark}
In some particular situations, alternative approximations of the integrals in the mild equations are possible,
see for instance \cite{MR2652783,MR2471778,MR3047942}. This could possibly lead to better numerical schemes or
improved error estimates, which will be investigated in future works.
\end{remark}

For the theoretical parts presented below, we will make use of the discrete Green function $G^M$ (see \eqref{eq:1})
in order to write the numerical scheme in a more suitable form.
We thus obtain the approximation $U_m^{n+1}\approx u(t_{n+1},x_m)$ given by (with a slight abuse of notations for the functions $f$ and $\sigma$)
\begin{align*}
U_m^{n+1}&=\frac{1}{M}\sum_{l=1}^{M-1}\sum_{j=1}^{M-1}\exp(\lambda_j^M\Delta t)\varphi_j(x_m)\varphi_j(x_l)U_l^n\\
&\quad+\Delta t\frac{1}{M}\sum_{l=1}^{M-1}\sum_{j=1}^{M-1}\exp(\lambda_j^M\Delta t)\varphi_j(x_m)\varphi_j(x_l)f(U_l^n)\\
&\quad+\frac{1}{\sqrt{M}}\sum_{l=1}^{M-1}\sum_{j=1}^{M-1}\exp(\lambda_j^M\Delta t)\varphi_j(x_m)\varphi_j(x_l)\sigma(U_l^n)(W_l^M(t_{n+1})-W_l^M(t_n)).
\end{align*}
The above equation can be written in the equivalent form
\begin{align*}
U_m^{n+1}&=\int_0^1 G^M(t_{n+1}-t_n,x_m,y)U_{M\kappa_M(y)}^n\,\dd y\\
&\quad+\int_{t_n}^{t_{n+1}}\int_0^1 G^M(t_{n+1}-t_n,x_m,y)f(U_{M\kappa_M(y)}^n)\,\dd y\,\dd s\\
&\quad+\int_{t_n}^{t_{n+1}}\int_0^1 G^M(t_{n+1}-t_n,x_m,y)\sigma(U_{M\kappa_M(y)}^n)\,W(\dd s,\dd y),
\end{align*}
where we recall that
$$G^M(t,x,y)=\sum_{j=1}^{M-1}\exp(\lambda_j^M t)\varphi_j^M(x)\varphi_j(\kappa_M(y)),$$
and $\kappa_M(y)=\frac{[My]}{M}$, $\varphi_j^M(x)=\varphi_j(x_l)$ for $x=x_l$ and $\varphi_j^M(x)=\varphi_j(x_l)+(Mx-l)(\varphi_j(x_{l+1})-\varphi_j(x_l))$ for $x\in(x_l, x_{l+1}].$
In order to exhibit a more convenient mild form of the numerical solution $U_m^n$, we iterate the integral equation
above to obtain
\begin{align*}
U_m^{n+1}&=\int_0^1 G^M(t_{n+1},x_m,y)u_0(\kappa_M(y))\,\dd y\\
&\quad+\sum_{r=0}^n\int_{t_r}^{t_{r+1}}\int_0^1 G^M(t_{n+1}-t_r,x_m,y)f(U_{M\kappa_M(y)}^r)\,\dd y\,\dd s\\
&\quad+\sum_{r=0}^n\int_{t_r}^{t_{r+1}}\int_0^1 G^M(t_{n+1}-t_r,x_m,y)\sigma(U_{M\kappa_M(y)}^r)\,W(\dd s,\dd y),
\end{align*}
for all $m=1,\dots,M-1$ and $n=0,1,\dots,N$. This implies that
\begin{align}
U_m^{n+1}&=\int_0^1 G^M(t_{n+1},x_m,y)u_0(\kappa_M(y))\,\dd y \nonumber \\
&\quad+\int_0^{t_{n+1}}\int_0^1 G^M(t_{n+1}-\kappa_N^T(s),x_m,y)
f\big(U_{M\kappa_M(y)}^{\kappa_N^T(s)/\Delta t}\big)\,\dd y\,\dd s \nonumber \\
&\quad+ \int_0^{t_{n+1}} \int_0^1 G^M(t_{n+1}-\kappa_N^T(s),x_m,y)\sigma\big(U_{M\kappa_M(y)}^{\kappa_N^T(s)/\Delta t}\big)\,W(\dd s,\dd y),
\label{eq:2}
\end{align}
where we have used the notation $\kappa_N^T(s):=T\kappa_N(\frac{s}{T})$. Set
$u^{M,N}(t_n,x_m):=U_m^n$. Then, equation \eqref{eq:2} yields
\begin{align}
&u^{M,N}(t_n,x_m)=\int_0^1 G^M(t_n,x_m,y)u_0(\kappa_M(y))\,\dd y \nonumber \\
&\quad+\int_0^{t_n}\int_0^1 G^M(t_n-\kappa_N^T(s),x_m,y)
f(u^{M,N}(\kappa_N^T(s),\kappa_M(y)))\,\dd y\,\dd s \nonumber \\
&\quad+ \int_0^{t_n} \int_0^1 G^M(t_n-\kappa_N^T(s),x_m,y)
\sigma(u^{M,N}(\kappa_N^T(s),\kappa_M(y)))\,W(\dd s,\dd y).
\label{eq:3}
\end{align}

At this point, we will introduce the {\it{weak}} form associated to the full discretization scheme, and
in particular to equation \eqref{eq:3}. This will allow us to define a continuous version of the scheme,
which will be denoted by $u^{M,N}(t,x)$, with $(t,x)\in [0,T]\times [0,1]$.
More precisely, let $\{v(t,x),\, (t,x)\in[0,T]\times [0,1]\}$ be the unique $\mathcal{F}_t$-adapted
continuous random field satisfying the following: for all $\Phi\in C^\infty(\mathbb{R}^2)$ with
$\Phi(t,0)=\Phi(t,1)=0$ for all $t$, it holds
\begin{align}
 \int_0^1 v(t,\kappa_M(y))\Phi(t,y) \dd y = & \int_0^1 u_0(\kappa_M(y))\Phi(t,y)\, \dd y \nonumber \\
 & \quad + \int_0^t\int_0^1 v(s,\kappa_M(y)) \left(\Delta_M \Phi(s,y) +
 \frac{\partial \Phi}{\partial s}(s,y)\right)\, \dd y\, \dd s \nonumber \\
 & \quad + \int_0^t\int_0^1 f(v(\kappa_N^T(s),\kappa_M(y))) \Phi(s,y)\, \dd y\, \dd s \nonumber \\
 & \quad + \int_0^t\int_0^1 \sigma(v(\kappa_N^T(s),\kappa_M(y))) \Phi(s,y)\, W(\dd s,\dd y),
 \quad \mathbb{P}\text{-a.s.},
 \label{eq:6}
\end{align}
for all $t\in [0,T]$. Here, $\Delta_M$ denotes the discrete
Laplacian, which is defined by, recalling that $\Delta x=\frac1M$,
\[
 \Delta_M\Phi(s,y):= (\Delta x)^{-2} \left\{\Phi(s,y+\Delta x) - 2\Phi(s,y) + \Phi(s,y-\Delta x)\right\}.
\]

Let us prove that, on the time-space grid points, the random field $v$ fulfills equation \eqref{eq:3}.
That is, we have the following result.

\medskip

\begin{lemma}\label{lem:1}
 With the above notations at hand, we have that, for all $m=1,\dots,M-1$ and
 $n=0,1,\dots,N$,
 \begin{align}
&v(t_n,x_m)=\int_0^1 G^M(t_n,x_m,y)u_0(\kappa_M(y))\,\dd y \nonumber \\
&\quad+\int_0^{t_n}\int_0^1 G^M(t_n-\kappa_N^T(s),x_m,y)
f(v(\kappa_N^T(s),\kappa_M(y)))\,\dd y\,\dd s \nonumber \\
&\quad+ \int_0^{t_n} \int_0^1 G^M(t_n-\kappa_N^T(s),x_m,y)
\sigma(v(\kappa_N^T(s),\kappa_M(y)))\,W(\dd s,\dd y).
\label{eq:4}
\end{align}
\end{lemma}

\begin{proof}
 We will follow some of the arguments developed in the proof of \cite[Thm.~3.2]{walsh1}. Indeed, for
 any $\phi\in C^\infty(\mathbb{R})$ and any $(t,y)\in [0,T]\times [0,1]$, we define
 \[
 G_t^M(\phi,y):=\int_0^1 G^M(t,z,y)\phi(z)\, \dd z.
 \]
 Since the Green function $G^M$ solves the discretized homogeneous heat equation with Dirichlet
 boundary conditions, that is, we have $G^M(t,x,0)=G^M(t,x,1)=0$ and, for any
 fixed $x\in (0,1)$,
 \[
  \frac{\partial}{\partial t} G^M(t,x,y)-\Delta_M G^M(t,x,y)=0,
 \]
we can infer that
\begin{align*}
 G_t^M(\phi,y) & = \int_0^1 \left(G^M(0,z,y)+\int_0^t \Delta_M G^M(s,z,y) \dd s\right) \phi(z)\, \dd z \\
 & = \int_0^1 G^M(0,z,y)\phi(z)\, \dd z + \int_0^t\int_0^1 \Delta_M G^M(s,z,y)\phi(z)\, \dd z\,\dd s.
\end{align*}
Hence $\displaystyle\frac{\partial}{\partial t}G_t^M(\phi,y)= \int_0^1 \Delta_M
G^M(s,z,y)\phi(z)\, \dd z$. On the other hand, since
\[\Delta_M
G_t^M(\phi,y)=\int_0^1 \Delta_M G^M(t,z,y)\phi(z)\, \dd z,\]
we deduce
that
\begin{equation}
 \frac{\partial}{\partial t}G_t^M(\phi,y) - \Delta_M G_t^M(\phi,y)=0,
\label{eq:5}
 \end{equation}
with $(t,y)\in [0,T]\times [0,1]$.

At this point, we take $\Phi(s,y)=G_{t-\kappa_N^T(s)}^M (\phi,y)$, with $t\in[0,T]$ and
$\phi\in C^\infty(\mathbb{R})$, and plug this $\Phi$ in \eqref{eq:6}. Thus, by \eqref{eq:5}
we get that
\begin{align*}
 \int_0^1 v(t,\kappa_M(y)) G_{t-\kappa_N^T(t)}^M (\phi,y)\,\dd y & =
 \int_0^1 u_0(\kappa_M(y)) G_t^M(\phi,y)\,\dd y \\
 & \quad + \int_0^t\int_0^1 f(v(\kappa_N^T(s),\kappa_M(y))) G_{t-\kappa_N^T(s)}^M (\phi,y)\, \dd y\, \dd s \nonumber \\
 & \quad + \int_0^t\int_0^1 \sigma(v(\kappa_N^T(s),\kappa_M(y))) G_{t-\kappa_N^T(s)}^M (\phi,y)\, W(\dd s,\dd y).
\end{align*}
Let $(\phi_\epsilon)_{\epsilon\geq 0}$ be an approximation of the Dirac delta $\delta_x$, for some
$x\in (0,1)$ (e.g. $\phi_\epsilon$ could be taken to be Gaussian kernels), so that we have
\begin{align*}
 \int_0^1 v(t,\kappa_M(y)) G_{t-\kappa_N^T(t)}^M (\phi_\epsilon,y)\,\dd y & =
 \int_0^1 u_0(\kappa_M(y)) G_t^M(\phi_\epsilon,y)\,\dd y \\
 & \quad + \int_0^t\int_0^1 f(v(\kappa_N^T(s),\kappa_M(y))) G_{t-\kappa_N^T(s)}^M (\phi_\epsilon,y)\, \dd y\, \dd s \nonumber \\
 & \quad + \int_0^t\int_0^1 \sigma(v(\kappa_N^T(s),\kappa_M(y))) G_{t-\kappa_N^T(s)}^M (\phi_\epsilon,y)\,
 W(\dd s,\dd y).
\end{align*}
Then, as it is done in the proof of \cite[Thm.~3.2]{walsh1}, take $\epsilon\rightarrow 0$ in the latter
equation, so we will end up with
\begin{align}
 & \int_0^1 G^M(t-\kappa_N^T(t),x,y) v(t,\kappa_M(y))\,   \dd y \nonumber \\
& \quad = \int_0^1 G^M(t,x,y) u_0(\kappa_M(y))\, \dd y \nonumber \\
 & \quad \quad + \int_0^t\int_0^1 G^M(t-\kappa_N^T(s),x,y) f(v(\kappa_N^T(s),\kappa_M(y)))\,  \dd y\, \dd s \nonumber \\
 & \quad \quad+ \int_0^t\int_0^1 G^M(t-\kappa_N^T(s),x,y) \sigma(v(\kappa_N^T(s),\kappa_M(y)))\,
 W(\dd s,\dd y).
 \label{eq:7}
\end{align}
Note that this equation, which is valid for any $(t,x)\in [0,T]\times [0,1]$, is very similar to
the one we would like to get, that is \eqref{eq:4}. In fact, taking $t=t_n$ and $x=x_m$ in
\eqref{eq:7} for some
$n\in \{0,\dots,N\}$ and $m\in \{1,\dots,M-1\}$, respectively, we have, using the explicit expression
of $G^M$,
\begin{align*}
 \int_0^1 G^M(0,x_m,y) v(t_n,\kappa_M(y))\, \dd y & =
 \int_0^1 \left(\sum_{j=1}^{M-1} \varphi_j(x_m) \varphi_j(\kappa_M(y))\right) v(t_n,\kappa_M(y))\, \dd y\\
 & = \sum_{j=1}^{M-1} \varphi_j(x_m) \int_0^1 \varphi_j(\kappa_M(y))  v(t_n,\kappa_M(y))\, \dd y \\
 & = \sum_{k=1}^{M-1} v(t_n,x_k) \frac1M \sum_{j=1}^{M-1} \varphi_j(x_m) \varphi_j(x_k)\\
 & = v(t_n,x_m),
\end{align*}
where in the last step we have applied \eqref{eq:8}. This concludes the lemma's proof.
\end{proof}

\medskip

As a consequence of Lemma \ref{lem:1}, comparing equations \eqref{eq:3} and \eqref{eq:4} we deduce that
$u^{M,N}(t_n,x_m)=v(t_n,x_m)$ for all $m=1,\dots,M-1$ and
 $n=0,1,\dots,N$. Thus, we can define a continuous version of $u^{M,N}$ as follows: for any
 $(t,x)\in [0,T]\times [0,1]$, set
 \[
  u^{M,N}(t,x):= \int_0^1 G^M(t-\kappa_N^T(t),x,y) v(t,\kappa_M(y))\,   \dd y.
 \]
Observe that, by \eqref{eq:7}, the random field $\{u^{M,N}(t,x),\, (t,x)\in [0,T]\times [0,1]\}$ satisfies
\begin{equation}
\begin{aligned}
u^{M,N}(t,x)&:=\int_0^1G^M(t,x,y)u_0(\kappa_M(y))\,\dd y\\
&\quad+\int_0^t\int_0^1 G^M(t-\kappa_N^T(s),x,y)f(u^{M,N}(\kappa_N^T(s),\kappa_M(y)))\,\dd y\,\dd s\\
&\quad+\int_0^t\int_0^1 G^M(t-\kappa_N^T(s),x,y)\sigma(u^{M,N}(\kappa_N^T(s),\kappa_M(y)))\,W(\dd s,\dd y).
\end{aligned}
\label{fullapp}
\end{equation}
The above mild form of the fully discrete approximation will be used in the proof of the main result
of the paper (see Theorem \ref{th:time}).

\medskip

\begin{remark}
 It can be easily proved that, if $t_n$ is any discrete time and $x\in (x_m,x_{m+1})$, then
 $u^{M,N}(t_n,x)$ turns out to be the linear interpolation between $u^{M,N}(t_n,x_m)$ and
 $u^{M,N}(t_n,x_{m+1})$. This is consistent with the definition of the space discrete
 approximation $u^M(t,x)$ whenever $x\in (x_m,x_{m+1})$ (see \eqref{eq:-1}).
\end{remark}


\subsubsection{Some properties of ${u^{M,N}}$}

This section is devoted to provide three results establishing
properties of the full approximation $u^{M,N}$ which will be needed
in the sequel.

\medskip

First, we note that the full approximation \eqref{fullapp} is bounded.
Indeed, the proof of the following proposition is very similar to that of Proposition~\ref{xreg} above
and is therefore omitted.

\smallskip

\begin{proposition}\label{fullappbdd}
Assume that $u_0\in C([0,1])$ with $u_0(0)=u_0(1)=0,$ and that the functions $f$ and $\sigma$ satisfy
the condition \eqref{LG}. Then, for every $p\geq 1$, there exists a constant $C$ such that
$$
\sup_{M,N\geq1}\sup_{(t,x)\in[0,T]\times[0,1]}\E[|u^{M,N}(t,x)|^{2p}]\leq C.
$$
\end{proposition}

Next, we define the following quantities:
$$
w^{M,N}(t,x):=u^{M,N}(t,x)-\int_0^1 G^M(t,x,y)u_0(\kappa_M(y))\,\dd y
$$
and
$$
w^M(t,x):=u^{M}(t,x)-\int_0^1 G^M(t,x,y)u_0(\kappa_M(y))\,\dd y,
$$
where we recall that $u^M$ stands for the spatial discretization introduced in Section \ref{sect:spatial}.
Then, we have the following result.

\smallskip

\begin{proposition}\label{holder}
Assume that $u_0\in C([0,1])$ with $u_0(0)=u_0(1)=0$, and that $f$ and $\sigma$ satisfy condition \eqref{LG}.
Then, for every $p\geq 1$, $t,r\in[0,T]$ and $x,z\in[0,1]$, we have
\begin{align}
&\E[|w^{M}(t,x)-w^{M}(r,z)|^{2p}]\leq C\left(|t-r|^{1/4}+|x-z|^{1/2}\right)^{2p}\label{holderM}\\
&\E[|w^{M,N}(t,x)-w^{M,N}(r,z)|^{2p}]\leq C\left(|t-r|^{1/4}+|x-z|^{1/2}\right)^{2p},\label{holderMN}
\end{align}
where the constant $C$ does not depend on $M$ neither on $N$.
\end{proposition}
\begin{proof}
Inequality \eqref{holderM} is proved in \cite[Prop.~3.7]{gyongy1}.
Let us now show inequality \eqref{holderMN}. By definition, we have
\begin{align*}
w^{M,N}(t,x)&=\int_0^t\int_0^1 G^M(t-\kappa_N^T(s),x,y)f(u^{M,N}(\kappa_N^T(s),\kappa_M(y)))\,\dd y\,\dd s\\
&\quad+\int_0^t\int_0^1 G^M(t-\kappa_N^T(s),x,y)\sigma(u^{M,N}(\kappa_N^T(s),\kappa_M(y)))\,W(\dd s,\dd y)\\
&=:F^{M,N}(t,x)+H^{M,N}(t,x),
\end{align*}
and hence
$$
w^{M,N}(t,x)-w^{M,N}(r,z)=F^{M,N}(t,x)-F^{M,N}(r,z)+H^{M,N}(t,x)-H^{M,N}(r,z).
$$
Therefore
\begin{align*}
\E[|w^{M,N}(t,x)-w^{M,N}(r,z)|^{2p}]&\leq C\bigl(\E[|F^{M,N}(t,x)-F^{M,N}(r,z)|^{2p}]\\
&\quad+\E[|H^{M,N}(t,x)-H^{M,N}(r,z)|^{2p}]\bigr).
\end{align*}
We will next prove that
$$
\E[|H^{M,N}(t,x)-H^{M,N}(r,z)|^{2p}]\leq C\left(|t-r|^{1/4}+|x-z|^{1/2}\right)^{2p}.
$$
The estimate for $F^{M,N}$ follows in a similar way.
We have
\begin{align*}
|H^{M,N}(t,x)-H^{M,N}(r,z)|^{2p}&\leq C\big(|H^{M,N}(t,x)-H^{M,N}(r,x)|^{2p}\\
&\quad+|H^{M,N}(r,x)-H^{M,N}(r,z)|^{2p}\bigr)
\end{align*}
and define
\begin{align*}
A^{2p}&:=\E[|H^{M,N}(t,x)-H^{M,N}(r,x)|^{2p}]\\
B^{2p}&:=\E[|H^{M,N}(r,x)-H^{M,N}(r,z)|^{2p}].
\end{align*}
Then $A^{2p}\leq C(A_1^{2p}+A_2^{2p})$, where, for $r\leq t$ without loss of generality,
\begin{align*}
A_1^{2p}&= \E\left[\left|\int_0^r\int_0^1(G^M(t-\kappa_N^T(s),x,y)-G^M(r-\kappa_N^T(s),x,y))\right.\right.\\
&\quad\times\left.\left.\sigma(u^{M,N}(\kappa_N^T(s),\kappa_M(y)))\, W(\dd s,\dd y)\right|^{2p}\right]\\
A_2^{2p}&= \E\left[\left|\int_r^t\int_0^1G^M(t-\kappa_N^T(s),x,y)\sigma(u^{M,N}(\kappa_N^T(s),\kappa_M(y)))\, W(\dd s,\dd y)\right|^{2p}\right].
\end{align*}
Using Burkholder-Davies-Gundy's inequality, Lemma \ref{greenreg}, assumption \eqref{LG} on $\sigma$,
Minkowski's inequality
and Proposition \ref{fullappbdd}, we have the estimates
\begin{align*}
A_1^2&=\left(\E\left[\left|\int_0^r\int_0^1(G^M(t-\kappa_N^T(s),x,y)-G^M(r-\kappa_N^T(s),x,y))\sigma(u^{M,N}(\kappa_N^T(s),\kappa_M(y))\, W(\dd s,\dd y)\right|^{2p}\right]\right)^{1/p}\\
&\leq C\left(\E\left[\left(\int_0^r\int_0^1|G^M(t-\kappa_N^T(s),x,y)-G^M(r-\kappa_N^T(s),x,y)|^2|\sigma(u^{M,N}(\kappa_N^T(s),\kappa_M(y)))|^2\, \dd y\,\dd s\right)^p\right]\right)^{1/p}\\
&=C\tnorm{\int_0^r\int_0^1|G^M(t-\kappa_N^T(s),x,y)-G^M(r-\kappa_N^T(s),x,y)|^2|\sigma(u^{M,N}(\kappa_N^T(s),\kappa_M(y)))|^2\,\dd y\,\dd s}_p\\
&\leq C\int_0^r\int_0^1|G^M(t-\kappa_N^T(s),x,y)-G^M(r-\kappa_N^T(s),x,y)|^2\tnorm{\sigma(u^{M,N}(\kappa_N^T(s),\kappa_M(y)))}_{2p}^2\,\dd y\,\dd s\\
&\leq C\int_0^r\int_0^1|G^M(t-\kappa_N^T(s),x,y)-G^M(r-\kappa_N^T(s),x,y)|^2\,\dd y\,\dd s\\
&\leq C(t-r)^{1/2},
\end{align*}
where we set $\tnorm{\cdot}_{2p}=\left(\E\left[|\cdot|^{2p}\right]\right)^{1/(2p)}$.
Using similar arguments we have
\begin{align*}
A_2^2&=\left(\E\left[\left|\int_r^t\int_0^1G^M(t-\kappa_N^T(s),x,y)\sigma(u^{M,N}(\kappa_N^T(s),\kappa_M(y))\, W(\dd s,\dd y)\right|^{2p}\right]\right)^{1/p}\\
&\leq C\left(\E\left[\left(\int_r^t\int_0^1|G^M(t-\kappa_N^T(s),x,y)|^2|\sigma(u^{M,N}(\kappa_N^T(s),\kappa_M(y)))|^2\, \dd y\,\dd s\right)^p\right]\right)^{1/p}\\
&\leq C\int_r^t\int_0^1|G^M(t-\kappa_N^T(s),x,y)|^2\tnorm{\sigma(u^{M,N}(\kappa_N^T(s),\kappa_M(y)))}_{2p}^2\,\dd y\,\dd s\\
&\leq C\int_r^t \frac{1}{(t-\kappa_N^T(s))^{1/2}}\,\dd s\\
&\leq C\int_r^t \frac{1}{(t-s)^{1/2}}\,\dd s\\
&\leq C(t-r)^{1/2}.
\end{align*}
Thus, we obtain
$$\E[|H^{M,N}(t,x)-H^{M,N}(r,x)|^{2p}]\leq C|t-r|^{p/2},$$
and we remark that this estimate is uniform with respect to $x\in [0,1]$.

\smallskip

It remains to estimate the term $B$. We have
\begin{align*}
B^{2p}:=\E\left[\left|\int_0^r\int_0^1(G^M(r-\kappa_N^T(s),x,y)-G^M(r-\kappa_N^T(s),z,y))\sigma(u^{M,N}(\kappa_N^T(s),\kappa_M(y)))\, W(\dd s,\dd y)\right|^{2p}\right],
\end{align*}
and estimating $B$ as we did for $A_1$ and $A_2$, we obtain
\begin{align*}
B^2&\leq C\int_0^r\int_0^1|G^M(r-\kappa_N^T(s),x,y)-G^M(r-\kappa_N^T(s),z,y)|^2\,\dd y\,\dd s\\
&\leq C\int_0^r\sum_{j=1}^{M-1} \exp(-2j^2\pi^2c_j^M(r-\kappa_N^T(s)))|\varphi_j^M(x)-\varphi_j^M(z)|^2\,\dd s\\
&\leq C\int_0^r\sum_{j=1}^{M-1} \exp(-2j^2\pi^2c_j^M(r-s))|\varphi_j^M(x)-\varphi_j^M(z)|^2\,\dd s.
\end{align*}
At this point, we note that the latter term also appears in the
proof of \cite[Lem.~3.6]{gyongy1}, so we can estimate it in the same
way and obtain
$$
\E[|H^{M,N}(r,x)-H^{M,N}(r,z)|^{2p}]\leq C|x-z|^{p},
$$
with a constant $C$ independent of $r$.
Collecting the estimates obtained so far we obtain the bound
\[
\E[|H^{M,N}(t,x)-H^{M,N}(r,z)|^{2p}]\leq C\left(|t-r|^{1/4}+|x-z|^{1/2}\right)^{2p},
\]
which finally leads to \eqref{holderMN}.
\end{proof}

\medskip

Finally, we shall also need the following regularity result
for the full approximation.

\smallskip

\begin{proposition}\label{regMN}
Assume that $f$ and $\sigma$ satisfy condition \eqref{LG}.
\begin{enumerate}
 \item If $u_0\in C([0,1])$ with $u_0(0)=u_0(1)=0$, then for any $s, t\in [0,T]$ and $x\in [0,1]$,
$p\geq 1$ and $\frac12<\alpha<\frac52$, we have
$$
\E[|u^{M,N}(t,x)-u^{M,N}(s,x)|^{2p}]\leq C s^{-\alpha p} |t-s|^{\tau p},
$$
 where $\tau = \frac12 \wedge(\alpha-\frac12)$ and with a constant $C$ independent of $M$, $N$ and $x$.
 \item If $u_0\in H^{\beta}([0,1])$, with $u_0(0)=u_0(1)=0$, for some
$\beta>\frac{1}{2}$, then for any $s, t\in [0,T]$ and $x,z\in [0,1]$,
and any $p\geq 1$, we have
$$
\E[|u^{M,N}(t,x)-u^{M,N}(s,z)|^{2p}]\leq C\bigl(|t-s|^{\tau
p}+|x-z|^{2 \tau p}\bigr),
$$
where $\tau = \frac12 \wedge(\beta-\frac12)$ and with a constant $C$ independent of $M$ and $N$.
\end{enumerate}
\end{proposition}

\begin{proof}
The proof can be built on the proof of Proposition \ref{reg}, so we will only sketch the main steps.

To start with, part 1 can be proved by following the same arguments used in the proof of part 1 of Proposition \ref{reg} and it is based on three estimates.
First, one applies that
\[
 \int_0^1 |G^M(t,x,y)-G^M(s,x,y)|^2\, \dd y \leq C s^{-\alpha} |t-s|^{\alpha-\frac12},
\]
which corresponds to part $(iii)$ in Lemma \ref{greenreg}. Secondly, we have
\[
 \int_s^t \int_0^1 |G^M(t-\kappa_N^T(r),x,y)|^2\, \dd y\, \dd r\leq C |t-s|^\frac12,
\]
which can be verified by using $(ii)$ of Lemma \ref{greenreg}. Finally, it holds that
\[
 \int_0^s \int_0^1 |G^M(t-\kappa_N^T(r),x,y)- G^M(s-\kappa_N^T(r),x,y)|^2\, \dd y\, \dd r \leq
 C |t-s|^\frac12.
\]
The latter estimate can be checked by doing some simple modifications in the proof of part $(i)$ in Lemma \ref{greenreg}.

As far as part 2 is concerned, the time increments can be analyzed following the same steps as
those used in the proof of part 2 in Proposition \ref{reg}. We will
sketch the proof for the spatial increments. More precisely, taking
into account equation \eqref{fullapp}, in order to control the term
$\E[|u^{M,N}(t,x)-u^{M,N}(t,z)|^{2p}]$ first we need to estimate the
expression
\[
\left|\int_0^1(G^M(t,x,y)-G^M(t,z,y))u_0(\kappa_M(y))\,\dd y\right|^{2p}.
\]
Using the same techniques as in the proof of part 2 in Proposition \ref{reg}, the above term
can be bounded by
\[
 \|u_0\|^{2p}_{H^\beta} \left|\sum_{j=1}^{M-1} j^{-2\beta} \big|\varphi_j^M(x)- \varphi_j^M(z)\big|^2 \right|^p, 
\]
where we recall that $\beta>\frac12$.
Next, it can be easily proved that $\big|\varphi_j^M(x)-
\varphi_j^M(z)\big|\leq C (1 \wedge j (z-x))$, where the constant
$C$ does not depend on $M$ and we have assumed, without loosing
generality, that $x<z$. Hence,
\[
 \left|\int_0^1(G^M(t,x,y)-G^M(t,z,y))u_0(\kappa_M(y))\,\dd y\right|^{2p} \leq C
 \left(\sum_{j=1}^\infty j^{-2\beta} (1 \wedge j^2 (z-x)^2)\right)^p.
\]
The latter series can be estimated, up to some constant, by $(z-x)^{(2\beta-1)p}$.

As far as the spatial increments of the remaining two terms in
equation \eqref{fullapp} is concerned, applying Burkholder-Davies-Gundy and
Minkowski's inequalities, as well as the linear growth on $f$ and
$\sigma$ and Proposition~\ref{fullappbdd}, the analysis reduces to
control the term
\[
 \left(\int_0^t\int_0^1 |G^M(t-\kappa_N^T(s),x,y)-G^M(t-\kappa_N^T(s),z,y)|^2\, \dd y\, \dd s\right)^p.
\]
The same arguments as above yield that this term can be bounded by
\begin{align*}
\left(\int_0^t \sum_{j=1}^{M-1} e^{2 \lambda_j^M (t-s)} \big(1
\wedge j^2(z-x)^2\big)\, \dd s\right)^p & \leq C \left(
\sum_{j=1}^\infty j^{-2} \big(1 \wedge j^2(z-x)^2\big)\right)^p
\\
& \leq C (z-x)^p.
\end{align*}
This concludes the proof.
\end{proof}

\medskip

\begin{remark}
Whenever $u_0\in H^{\beta}([0,1])$ for some
$\beta>\frac{1}{2}$, the above result implies, thanks to Kolmogorov's continuity
criterion, that the random field $u^{M,N}$ has a version with
H\"older-continuous sample paths.
\end{remark}

\subsubsection{Main result}

We are now ready to formulate and prove the main result of this
section. Recall that $u^M$ is the space discrete approximation given
by \eqref{spaceapp} and $u^{M,N}$ is the full discretization given
by \eqref{fullapp}.

\begin{theorem}\label{th:time}
Assume that $f$ and $\sigma$ satisfy the conditions \eqref{L} and
\eqref{LG}.
\begin{enumerate}
\item If $u_0\in C([0,1])$ with $u_0(0)=u_0(1)=0$, then for any $p\geq
1$, $0<\mu<\frac14$ and $t\in [0,T]$, there exists a constant
$C=C(p,\mu,t)$ such that
$$
\sup_{x\in
[0,1]}\Big(\E[|u^{M,N}(t,x)-u^M(t,x)|^{2p}]\Big)^{\frac{1}{2p}}\leq
C(\Delta t)^\mu.
$$
\item If $u_0\in H^{\beta}([0,1])$ for some $\beta>\frac{1}{2}$,
with $u_0(0)=u_0(1)=0$, then for any $p\geq 1$, we have
$$
\sup_{t\in[0,T]}\sup_{x\in
[0,1]}\Big(\E[|u^{M,N}(t,x)-u^M(t,x)|^{2p}]\Big)^{\frac{1}{2p}}\leq
C(\Delta t)^\nu,
$$
where $\nu=\frac14 \wedge(\frac{\beta}{2}-\frac14)$.
\end{enumerate}
\end{theorem}

\begin{proof}
We have, using the notation $\tnorm{\cdot}_{2p}=\left(\E\left[|\cdot|^{2p}\right]\right)^{1/(2p)}$,
\begin{align*}
&\tnorm{u^{M,N}(t,x)-u^M(t,x)}_{2p}\\
&\leq\tnorm{\int_0^t\int_0^1\left(G^M(t-\kappa_N^T(s),x,y)f(u^{M,N}(\kappa_N^T(s),\kappa_M(y)))\right.\\
&\quad\quad\left. -G^M(t-s,x,y)f(u^M(s,\kappa_M(y)))\right)\,\dd y\,\dd s}_{2p}\\
&\quad+\tnorm{\int_0^t\int_0^1\left(G^M(t-\kappa_N^T(s),x,y)\sigma(u^{M,N}(\kappa_N^T(s),\kappa_M(y)))\right.\\
&\quad\quad\left. -G^M(t-s,x,y)\sigma(u^M(s,\kappa_M(y)))\right)\, W(\dd s,\dd y)}_{2p}\\
&=: A+B.
\end{align*}

We show in detail the estimates for $B$. It will then be clear that
similar estimates can be made for $A$. First we note that
\begin{align*}
B^2 &\leq C(B_1^2+B_2^2),
\end{align*}
where
\begin{align*}
B_1^2&=\tnorm{\int_0^t\int_0^1 (G^M(t-\kappa_N^T(s),x,y)-G^M(t-s,x,y))\sigma(u^{M,N}(\kappa_N^T(s),\kappa_M(y)))\, W(\dd s,\dd y)}_{2p}^2
\end{align*}
and
\begin{align*}
B_2^2 &=\tnorm{\int_0^t\int_0^1 G^M(t-s,x,y)(\sigma(u^{M,N}(\kappa_N^T(s),\kappa_M(y)))-\sigma(u^M(s,\kappa_M(y))))\, W(\dd s,\dd y)}_{2p}^2.
\end{align*}

By Burkholder-Davies-Gundy and Minkowski's inequalities, we have
\begin{align*}
B_1^2 &= \left(\E\left[\left|\int_0^t\int_0^1 (G^M(t-\kappa_N^T(s),x,y)-G^M(t-s,x,y))\sigma(u^{M,N}(\kappa_N^T(s),\kappa_M(y)))\, W(\dd s,\dd y)\right|^{2p}\right]\right)^{1/p}\\
&\leq C\left(\E\left[\left(\int_0^t\int_0^1 |G^M(t-\kappa_N^T(s),x,y)-G^M(t-s,x,y)|^2|\sigma(u^{M,N}(\kappa_N^T(s),\kappa_M(y)))|^2\, \dd y\,\dd s\right)^p\right]\right)^{1/p}\\
&= C\tnorm{\int_0^t\int_0^1|G^M(t-\kappa_N^T(s),x,y)-G^M(t-s,x,y)|^2|\sigma(u^{M,N}(\kappa_N^T(s),\kappa_M(y)))|^2\, \dd y\,\dd s}_{p}\\
&\leq C\int_0^t\int_0^1 |G^M(t-\kappa_N^T(s),x,y)-G^M(t-s,x,y)|^2\tnorm{\sigma(u^{M,N}(\kappa_N^T(s),\kappa_M(y)))}_{2p}^2\,\dd y\,\dd s.
\end{align*}
By assumption \eqref{LG} and Proposition~\ref{fullappbdd}, we obtain
\begin{align*}
B_1^2&\leq\sup_{(s,y)\in [0,T]\times[0,1]}\tnorm{\sigma(u^{M,N}(s,y))}_{2p}^2\\
&\quad\quad\times\int_0^t\int_0^1 |G^M(t-\kappa_N^T(s),x,y)-G^M(t-s,x,y)|^2\,\dd y\,\dd s\\
&\leq C(\Delta t)^{1/2}.
\end{align*}
Here we have also used that
$$\sup_{x\in [0,1]} \int_0^t\int_0^1 |G^M(t-\kappa_N^T(s),x,y)-G^M(t-s,x,y)|^2\,\dd y\,\dd s\leq C(\Delta t)^{1/2},$$
where the constant $C$ does not depend on $M$. This is only a slight
variation of \eqref{greenreg1} in Lemma \ref{greenreg}. The proof is
very similar and is therefore omitted.

Concerning the term $B_2$, using analogous arguments we have
\begin{align*}
B_2^2&\leq C\int_0^t\int_0^1|G^M(t-s,x,y)|^2\,\dd y\\
&\quad\quad\times\sup_{y\in[0,1]}\tnorm{\sigma(u^{M,N}(\kappa_N^T(s),y))-\sigma(u^M(s,y))}_{2p}^2\,\dd s.
\end{align*}
By the Lipschitz assumption on $\sigma$ and $(ii)$ in
Lemma~\ref{greenreg}, we get
\begin{align}
B_2^2&\leq C\int_0^t\int_0^1 |G^M(t-s,x,y)|^2\,\dd y\sup_{x\in[0,1]}\tnorm{u^{M,N}(\kappa_N^T(s),x)-u^M(s,x)}_{2p}^2\,\dd s \nonumber \\
&\leq C\int_0^t\frac{1}{\sqrt{t-s}}\left(\sup_{x\in[0,1]}\tnorm{u^{M,N}(\kappa_N^T(s),x)-u^{M,N}(s,x)}_{2p}^2\right. \nonumber\\
&\quad\quad\left. +\sup_{x\in[0,1]}\tnorm{u^{M,N}(s,x)-u^M(s,x)}_{2p}^2\right)\,\dd s \nonumber \\
&\leq C\int_0^t\frac{1}{\sqrt{t-s}} \sup_{x\in[0,1]}\tnorm{u^{M,N}(\kappa_N^T(s),x)-u^{M,N}(s,x)}_{2p}^2 \, \dd s \nonumber \\
&\quad\quad + C \int_0^t\frac{1}{\sqrt{t-s}} \sup_{x\in[0,1]}\tnorm{u^{M,N}(s,x)-u^M(s,x)}_{2p}^2\,\dd s.
\label{eq:9}
\end{align}
At this point, We need to distinguish between the two different cases of the initial value $u_0$.

If we assume $u_0\in C([0,1])$, then we apply Proposition \ref{regMN} to the first term in \eqref{eq:9}, so we get
\begin{align*}
\int_0^t\frac{1}{\sqrt{t-s}} \sup_{x\in[0,1]}\tnorm{u^{M,N}(\kappa_N^T(s),x)-u^{M,N}(s,x)}_{2p}^2 \, \dd s
& \leq C (\Delta t)^\tau \int_0^t (t-s)^{-\frac12} s^{-\alpha}\, \dd s \\
& = C (\Delta t)^\tau \, B\Big(1-\alpha,\frac12\Big)\, t^{\frac12-\alpha},
\end{align*}
where $B$ denotes the Beta function. In order to obtain the last equality, we need to restrict the range on
$\alpha$ to $(\frac12,1)$ (part 1 in Proposition~\ref{regMN} was valid for any $\alpha \in (\frac12,\frac52)$).
In this case, notice that we have $\tau=\frac12 \wedge (\alpha-\frac12)=\alpha-\frac12$. Plugging the above estimate in \eqref{eq:9} and taking into account
that  we obtained the bound $B_1^2\leq C (\Delta t)^\frac12$, we have thus proved that
\[
 B^2 \leq C(t) (\Delta t)^{\alpha-\frac12} +  C \int_0^t\frac{1}{\sqrt{t-s}} \sup_{x\in[0,1]}\tnorm{u^{M,N}(s,x)-u^M(s,x)}_{2p}^2\,\dd s.
\]
As commented at the beginning of the proof, the analysis of the term
$A^2$ can be performed in a similar way, in such a way that the same
type of estimate can be obtained. Summing up, we have that
\[
 z(t) \leq C(t) (\Delta t)^{\alpha-\frac12} +  C \int_0^t\frac{1}{\sqrt{t-s}} z(s)\,\dd s,
\]
where $z(s):=\sup_{x\in [0,1]}\tnorm{u^{M,N}(s,x)-u^M(s,x)}_{2p}^2$. Then, applying a version 
of Gronwall's Lemma (see for instance \cite[Chap.~1]{pachpatte06}) we conclude this part of the proof.

\smallskip

If we instead assume $u_0\in H^\beta([0,1])$ for some $\beta>\frac{1}{2}$,
then we apply part 2 of Proposition~\ref{regMN} to the first term in \eqref{eq:9}, obtaining
\[
 \int_0^t\frac{1}{\sqrt{t-s}} \sup_{x\in[0,1]}\tnorm{u^{M,N}(\kappa_N^T(s),x)-u^{M,N}(s,x)}_{2p}^2 \, \dd s
 \leq C (\Delta t)^\tau,
\]
where $\tau=\frac12 \wedge (\beta-\frac12)$. Hence, in this case we get that
\[
 z(t) \leq C (\Delta t)^\tau +  C \int_0^t\frac{1}{\sqrt{t-s}} z(s)\,\dd s,
\]
and we conclude applying again a version of Gronwall's Lemma, see for instance \cite[Lem.~3.4]{gyongy1}.
%
\end{proof}

\medskip

Combining Theorems~\ref{th:space} and~\ref{th:time}, we arrive at
the following error estimate for the full discretization.
\begin{theorem}\label{fullLip}
Let $f$ and $\sigma$ satisfy conditions \eqref{L} and \eqref{LG}.
\begin{enumerate}
\item Assume that $u_0\in C([0,1])$ with $u_0(0)=u_0(1)=0$.
Then, for every $p\geq 1$, $t\in (0,T]$, $0<\alpha_1<\frac{1}{4}$
and $0<\alpha_2<\frac14$, there are constants $C_i=C_i(t)$, $i=1,2$,
such that
\begin{align*}
\sup_{x\in[0,1]}\left(
\E[|u^{M,N}(t,x)-u(t,x)|^{2p}]\right)^{\frac{1}{2p}} \leq C_1
(\Delta x)^{\alpha_1}+C_2 (\Delta t)^{\alpha_2}.
\end{align*}
\item Assume that $u_0\in H^\beta([0,1])$
with $u_0(0)=u_0(1)=0$, for some $\beta>\frac{1}{2}$. Then, for
every $p\geq 1$, $t\in (0,T]$, $0<\alpha_1<\frac{1}{4}$, there are
constants $C_1=C_1(t)$ and $C_2$ such that
\begin{align*}
\sup_{x\in[0,1]}\left(\E[|u^{M,N}(t,x)-u(t,x)|^{2p}]\right)^{\frac{1}{2p}}\leq
C_1(\Delta x)^{\alpha_1}+C_2(\Delta t)^\tau,
\end{align*}
where $\tau=\frac14\wedge(\frac{\beta}{2}-\frac14)$.
\end{enumerate}
\end{theorem}

\medskip

\begin{remark}
For ease of presentation, we stated the above results for functions $f$ and $\sigma$ depending only on $u$.
Observe that the above results remain true in the case of functions $f$ and $\sigma$ depending on $(t,x,u)$ if one replaces
the condition \eqref{L} by the following one
\begin{align}\label{H}
|f(t,x,u)-f(s,y,v)|+|\sigma(t,x,u)-\sigma(s,y,v)|\leq C\bigl(|t-s|^{1/4}+|x-y|^{1/2}+|u-v|\bigr)\tag{H}
\end{align}
for all $s,t\in[0,T]$, $x,y\in[0,1]$, $u,v\in\R$. In this case, the fully discrete solution reads
\begin{equation*}
\begin{aligned}
u^{M,N}(t,x)&=\int_0^1G^M(t,x,y)u_0(\kappa_M(y))\,\dd y\\
&\quad+\int_0^t\int_0^1 G^M(t-\kappa_N^T(s),x,y)f(\kappa_N^T(s),\kappa_M(y),u^{M,N}(\kappa_N^T(s),\kappa_M(y)))\,\dd y\,\dd s\\
&\quad+\int_0^t\int_0^1 G^M(t-\kappa_N^T(s),x,y)\sigma(\kappa_N^T(s),\kappa_M(y),u^{M,N}(\kappa_N^T(s),\kappa_M(y)))\,W(\dd s,\dd y),
\end{aligned}
\end{equation*}
where we recall that $\kappa_M=\frac{[My]}{M}$ and $\kappa_N^T(s)=T\kappa_N(\frac{s}{T})$.
\end{remark}

\subsubsection{Numerical experiments: strong convergence}\label{sect-numexpstrong}

We now numerically illustrate the results from
Theorem~\ref{th:time}. To do so, we first discretize the problem
\eqref{heateq}, with $u_0(x)=\cos(\pi(x-1/2))$, $f(u)=u/2$,
$\sigma(u)=1-u$ with centered finite differences using the mesh
$\Delta x=2^{-9}$. The time discretizations are done using the
semi-implicit Euler-Maruyama scheme (see e.g. \cite{gyongy2}), the
Crank-Nicolson-Maruyama scheme (see e.g. \cite{Walsh}) and the
explicit exponential integrator \eqref{sexp} with step sizes $\Delta
t$ ranging from $2^{-1}$ to $2^{-16}$. The loglog plots of the
errors
$\sup_{(t,x)\in[0,0.5]\times[0,1]}\E[|u^{M,N}(t,x)-u^M(t,x)|^2]$ are
shown is Figure~\ref{fig:strong}, where convergence of order $1/2$
for the exponential integrator is observed. The reference solution
is computed with the exponential integrator using $\Delta
x_{\text{ref}}=2^{-9}$ and $\Delta t_{\text{ref}}=2^{-16}$. The
expected values are approximated by computing averages over
$M_s=500$ samples.

\begin{figure}
\begin{center}
\includegraphics*[height=7cm,keepaspectratio]{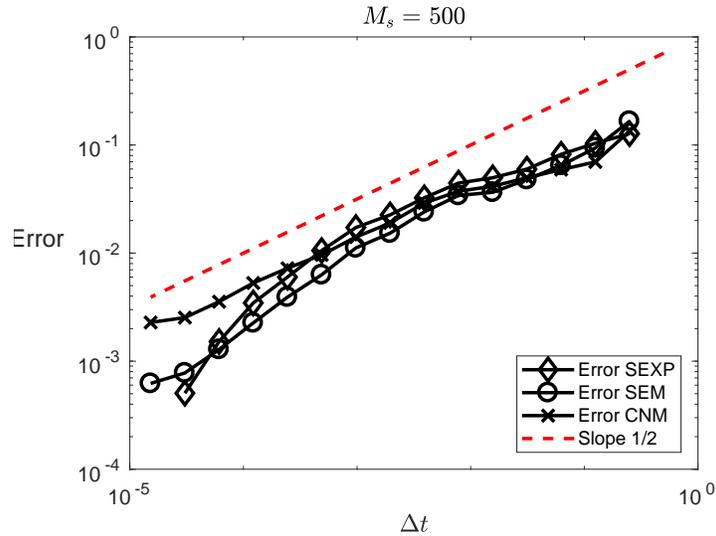}
\caption{Temporal rates of convergence  for the exponential integrator (SEXP), the semi-implicit Euler-Maruyama scheme (SEM),
and the Crank-Nicolson-Maruyama scheme (CNM).
The reference line has slope $1/2$ (dashed line).}
\label{fig:strong}
\end{center}
\end{figure}

Next, we compare the computational costs of the explicit stochastic
exponential method \eqref{sexp}, the semi-implicit Euler-Maruyama
scheme, and the Crank-Nicolson-Maruyama scheme for the numerical
integration of problem \eqref{heateq} with the same parameters as in
the previous numerical experiments. We run the numerical methods
over the time interval $[0,1]$. We discretize the spatial domain
$[0,1]$ with a mesh $\Delta x=2^{-6}$. We run $100$ samples for each
numerical method. For each method and each sample, we run several
time steps and compare the error at final time with a reference
solution provided for the same sample with the same method for the
very small time step $\Delta t_{\text{ref}}=2^{-15}$.
Figure~\ref{fig:compcost} shows the total computational time for all
the samples, for each method and each time step, as a function of
the averaged final error we obtain.

\begin{figure}
\begin{center}
\includegraphics*[height=6cm,keepaspectratio]{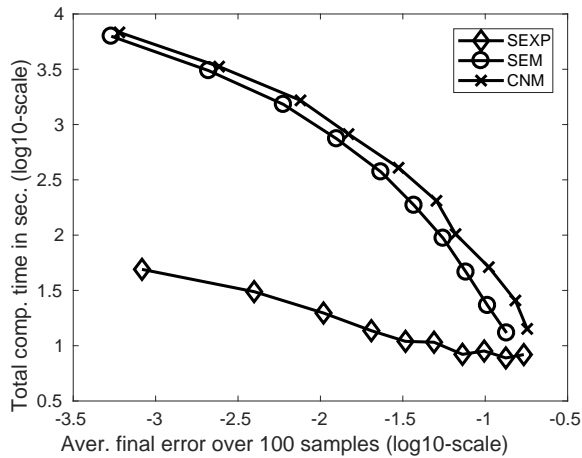}
\caption{Computational time as a function of the averaged final error
for the following numerical methods: the stochastic exponential scheme \eqref{sexp} (SEXP),
the semi-implicit Euler-Maruyama (SEM),
and the Crank-Nicholson-Maruyama scheme (CNM).}
\label{fig:compcost}
\end{center}
\end{figure}

We observe that the computational cost of the Crank-Nicolson-Maruyama scheme is slightly higher than
the cost of the semi-implicit Euler-Maruyama scheme which is a little bit higher
than the one for the explicit scheme \eqref{sexp}.

\subsection{Full discretization: almost sure convergence}

In this subsection we prove almost sure convergence of the fully
discrete approximation $u^{M,N}$ \eqref{fullapp} to the exact solution $u$ of the stochastic heat equation \eqref{heateq}
with globally Lipschitz continuous coefficients. The main result is the following.

\smallskip

\begin{theorem}\label{th:as}
Assume that the functions $f$ and $\sigma$ satisfy the conditions \eqref{LG} and \eqref{L},
and that $u_0\in C([0,1])$ with $u_0(0)=u_0(1)=0$.
Then, the full approximation $u^{M,N}(t,x)$ converges to $u(t,x)$
almost surely, as $M,N\rightarrow \infty$, uniformly in $t\in[0,T]$ and $x\in[0,1]$.
\end{theorem}

\begin{proof}
In \cite[Thm.~3.1]{gyongy1}, it was shown that $u^M(t,x)$
converges to $u(t,x)$ almost surely uniformly in $(t,x)$ as $M\rightarrow\infty$.
It is therefore enough to show that $u^{M,N}(t,x)$ converges to $u^M(t,x)$ almost surely,
as $N\rightarrow\infty$,
uniformly in $(t,x)$ and $M\in\mathbb{N}$. To achieve this, it suffices to prove that $w^{M,N}(t,x)$
converges to
$w^M(t,x)$ almost surely in $(t,x)$ as $N\rightarrow\infty$.
This is because the terms involving $u_0$ in the approximations
$u^M$ given by \eqref{spaceapp} and $u^{M,N}$ given by \eqref{fullapp} are the same.
We first observe that
$$
|w^{M,N}(t,x)-w^M(t,x)|^{2p}\leq C(A_1+A_2+A_3),
$$
where
\begin{align*}
A_1&=\sum_{n=0}^N\sum_{i=0}^N\left|w^{M,N}(t_n,x_i)-w^M(t_n,x_i)\right|^{2p}\\
A_2&=\sup_{n=0,\ldots,N}\sup_{i=0,\ldots,N}\sup_{|x-x_i|\leq 1/N}\sup_{|t-t_n|\leq \Delta t}\left|w^{M,N}(t,x)-w^{M,N}(t_n,x_i)\right|^{2p}\\
A_3&=\sup_{n=0,\ldots,N}\sup_{i=0,\ldots,N}\sup_{|x-x_i|\leq 1/N}\sup_{|t-t_n|\leq \Delta t}\left|w^{M}(t,x)-w^{M}(t_n,x_i)\right|^{2p}
\end{align*}
and we recall that $x_i$ and $t_n$ are the discrete points in space and time, respectively,
given by $x_i=\frac{i}{N}$ for $i=0,1\ldots,N$ and $t_n=n\Delta t$ for $n=0,1,\ldots,N$.
By Theorem~\ref{th:time} we obtain
$$
\E[A_1]\leq C\left(\frac{1}{N}\right)^{2\mu p-2},
$$
for all $0<\mu<\frac14$.
Also, by Proposition~\ref{holder} we have
$$
\E[A_2+A_3]\leq C\left(\frac{1}{N}\right)^{2p\delta}
$$
for $\delta\in(0,1/4)$. Using that
$$
\left(\frac{1}{N}\right)^{2\mu p-2}+\left(\frac{1}{N}\right)^{2p\delta}\leq 2\left(\frac{1}{N}\right)^{2 p\min(\delta,\mu)-2}
$$
we thus get
$$
\E\left[\sup_{M\geq 1}\sup_{(t,x)\in[0,T]\times [0,1]}|w^{M,N}(t,x)-w^M(t,x)|^{2p}\right]\leq C\left(\frac{1}{N}\right)^{2p\min(\delta,\mu)-2},
$$
where the constant $C$ does not depend on $M$ neither on $N$.
Hence, using Markov's inequality we obtain that
$$
\mathbb{P}\left(\sup_{M\geq 1}\sup_{(t,x)\in[0,T]\times[0,1]}|w^{M,N}(t,x)-w^M(t,x)|^{2p}>\left(\frac{1}{N}\right)^2\right)\leq
C\left(\frac{1}{N}\right)^{2 p\min(\delta,\mu)-4}
$$
for all integers $N\geq 1$. It thus follows that
$$
\sum_{N=1}^\infty \mathbb{P}\left(\sup_{M\geq 1}\sup_{(t,x)\in[0,T]\times[0,1]}|w^{M,N}(t,x)-w^M(t,x)|^{2p}>\left(\frac{1}{N}\right)^2\right)<\infty
$$
for $p$ large enough. By the Borel-Cantelli lemma we now know that for sufficiently large $p$ we have
$$
\sup_{M\geq 1}\sup_{(t,x)\in[0,T]\times[0,1]}|w^{M,N}(t,x)-w^M(t,x)|^{2p}\leq \frac{1}{N^2},
$$
with probability one. Taking the limit $N\to\infty$ concludes the proof.
\end{proof}

\subsubsection{Numerical experiments: almost sure convergence}\label{sect-numexpAS}

We now numerically illustrate Theorem~\ref{th:as}. To do so, we
first discretize the stochastic heat equation \eqref{heateq}, with
$u_0(x)=\cos(\pi(x-1/2))$, $f(u)=1-u$, $\sigma(u)=\sin(u)$ with
centered finite differences using the mesh $\Delta x=2^{-9}$. The
time discretization is done using the explicit exponential
integrator \eqref{sexp} with step sizes $\Delta t$ ranging from
$2^{-6}$ to $2^{-18}$ (only every second power). Figure~\ref{fig:as}
displays, for a fixed spatial discretization, profiles of one
realization of the numerical solution at the fixed time $T=0.5$ as
well as a reference solution computed with the exponential
integrator using $\Delta x_{\text{ref}}=2^{-9}$ and $\Delta
t_{\text{ref}}=2^{-18}$. Convergence to this reference solution as
the time step goes to zero (from light to dark grey plots) is
observed.

\begin{figure}
\begin{center}
\includegraphics*[height=7cm,keepaspectratio]{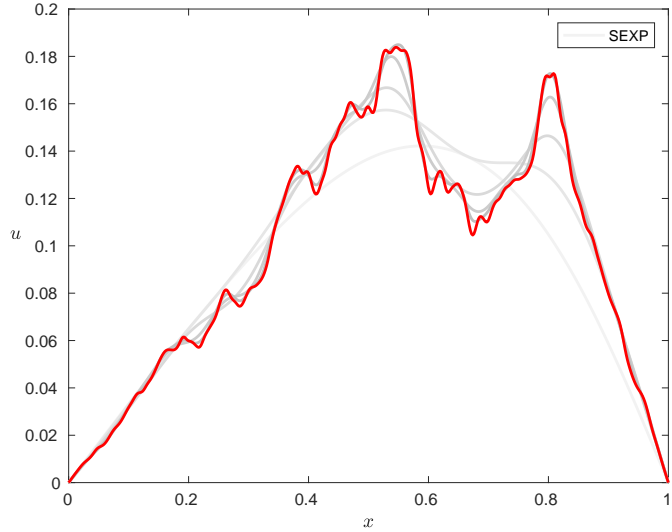}
\caption{Almost sure convergence of the exponential integrator (SEXP). The reference solution is displayed in red.}
\label{fig:as}
\end{center}
\end{figure}

\section{Convergence analysis for non-globally Lipschitz continuous coefficients}\label{section:nonLipschitz}

In this section, we remove the globally Lipschitz assumption on the
coefficients $f$ and $\sigma$ in equation \eqref{heateq} and we
prove convergence in probability of the fully discrete approximation
$u^{M,N}$ given by \eqref{fullapp} to the exact solution $u$ of
\eqref{heateq}. Throughout the section we will assume that the
initial condition $u_0$ belongs to $H^\beta$ for some
$\beta>\frac12$.

\smallskip

Furthermore, we shall consider the following hypotheses:
\begin{itemize}
\item[(PU)] Pathwise uniqueness holds for problem \eqref{heateq}: whenever $u$ and $v$ are carried by
the same filtered probability space and if both $u$ and $v$ are
solutions to problem \eqref{heateq} on the stochastic time interval
$[0,\tau)$, then $u(t,\cdot)=v(t,\cdot)$ for all $t\in[0,\tau)$,
almost surely.
\item[(C)] The coefficient functions $f(t,x,u)$ and $\sigma(t,x,u)$ are continuous in the variable $u$. 
\end{itemize}

\smallskip

\begin{remark}
For general conditions ensuring pathwise uniqueness in equation
\eqref{heateq}, we refer the reader to \cite{GP1,GP2}. Nevertheless,
note that pathwise uniqueness for parabolic stochastic partial
differential equations is an active research topic. Indeed, we
mention, for instance, the works \cite{MR1608641} (Lipschitz
coefficients), \cite{MR2773025,MR3357612} (H\"older coefficients),
\cite{MR3127884,MR3422943} (additive noise), where this question is
investigated. These results provide examples of parabolic stochastic
partial differential equations where assumption (PU) is fulfilled.
\end{remark}

\smallskip

In order to prove the main result of the section (cf Theorem~\ref{th:proba}), 
we will follow a similar approach as in
\cite{gyongy1} (see also \cite{ps05}). More precisely, we will first
use the results from Section~\ref{section:Lipschitz} to deduce that
the family of laws determined by $u^{M,N}$ are tight in the space of
continuous functions. Then, we will apply Skorokhod's representation
theorem and make use of the weak form \eqref{eq:6} corresponding to
the fully discrete approximation $u^{M,N}$. Finally, a suitable
passage to the limit and assumption (PU) will let us
conclude the proof.

\smallskip

We will use the above strategy in a successful way thanks the
following two auxiliary results.

\begin{lemma}[Lemma $4.5$ in \cite{gyongy1}]\label{limit}
For all $k\geq 0$, let $z^k=\{z^k(t,x)\colon t\geq 0, x\in[0,1]\}$
be a continuous $\mathcal{F}_t^k$-adapted random field and let
$W^k=\{W^k(t,x)\colon t\geq 0, x\in[0,1]\}$ be a Brownian sheet
carried by some filtered probability space $(\Omega, \mathcal{F},
(\mathcal{F}_t^k)_{t\geq 0},P)$. Assume also that, for every
$\epsilon>0$
$$
\lim_{k\rightarrow\infty}P\left(\sup_{t\in[0,T]}\sup_{x\in[0,1]}(|z^k-z^0|+|W^k-W^0|)(t,x)\geq \epsilon\right)=0.
$$
Let $h=h(t,x,r)$ be a bounded Borel function of
$(t,x,r)\in\mathbb{R}_+\times[0,1]\times\mathbb{R}$, which is
continuous in $r\in\mathbb{R}$. Then, letting $k\rightarrow \infty$,
\begin{align*}
&\int_0^t\int_0^1 h(s,x,z^k(s,x))\,\dd x\,\dd s\longrightarrow \int_0^t\int_0^1 h(s,x,z^0(s,x))\,\dd x\,\dd s,\\
&\int_0^t\int_0^1 h(s,x,z^k(s,x))\, W^k(\dd s,\dd x)\longrightarrow \int_0^t\int_0^1 h(s,x,z^0(s,x))\, W^0(\dd s,\dd x),
\end{align*}
in probability for every $t\in[0,T]$.
\end{lemma}

\smallskip

\begin{lemma}[Lemma $4.4$ in \cite{gyongy1}]\label{polish}
Let $E$ be a Polish space equipped with the Borel $\sigma$-algebra.
A sequence of $E$-valued random elements $(z_n)_{n\geq 1}$ converges
in probability if and only if, for every pair of subsequences
$z_l:=z_{n_l}$ and $z_m:=z_{n_m}$, there exists a subsequence
$v_k:=(z_{l_k}, z_{m_k})$ converging weakly to a random element $v$
supported on the diagonal $\{(x,y)\in E\times E\colon x=y\}$.
\end{lemma}

\smallskip

We are now ready to state and prove the main result of this section.

\smallskip

\begin{theorem}\label{th:proba}
Assume that the coefficients $f$ and $\sigma$ satisfy condition
\eqref{LG}, and that hypotheses (PU) and (C) are fulfilled. Then, there exists a random field
$u=\{u(t,x)\colon t\geq 0, x\in[0,1]\}$ such that, for every
$\epsilon>0$,
$$
\mathbb{P}\left(\sup_{t\in[0,T]}\sup_{x\in[0,1]}|u^{M_k,N_k}(t,x)-u(t,x)|\geq\epsilon\right)\rightarrow
0,
$$
as $k$ tends to infinity, for all sequences of positive integers
$(M_k,N_k)_{k\geq 1}$ such that $M_k, N_k\rightarrow\infty$, as
$k\rightarrow \infty$, where we recall that $u^{M,N}$ denotes the
fully discrete solution \eqref{fullapp}. Furthermore, the random
field $u$ is the unique solution to the stochastic heat equation
\eqref{heateq}.
\end{theorem}

\begin{proof}
We first show that the sequence $(u^{M,N})_{M,N\geq1}$ defines a
tight family of laws in the space $C([0,T]\times[0,1])$. To do so,
we invoke part 2 in Proposition \ref{regMN} on the regularity of the
numerical solution and we apply the tightness criterion on the plane
\cite[Thm.~2.2]{MR2678391}, which generalizes a well-known result of
Billingsley. Furthermore, Prokhorov's theorem implies that the
sequence of laws $(u^{M,N})_{M,N\geq1}$ is relatively compact in
$C([0,T]\times[0,1])$.

\smallskip
Fix any pair of sequences $(M_k,N_k)_{k\geq 1}$ such that $M_k,
N_k\rightarrow\infty$, as $k\rightarrow \infty$. Then, the laws of
$v_k:=u^{M_k,N_k}$, $k\geq 1$, form a tight family in the space
$C([0,T]\times[0,1])$.

\smallskip

Let now $(v^1_j)_{j\geq 1}$ and $(v^2_\ell)_{\ell\geq 1}$ be two
subsequences of $(v_k)_{k\geq 1}$. By Skorokhod's Representation
Theorem, there exist subsequences of positive integers $(j_r)_{r\geq
1}$ and $(\ell_r)_{r\geq 1}$ of the indices $j$ and $\ell$, a
probability space
$(\widehat\Omega,\widehat{\mathcal{F}},(\widehat{\mathcal{F}_t})_{t\geq1},\widehat{\mathbb{P}})$,
and a sequence of continuous random fields $(z_r)_{r\geq 1}$ with
$z_r:=\bigl(\widetilde u_r,\overline u_r,\widehat W_r\bigr)$, $r\geq1$, such that

\smallskip

\begin{enumerate}
\item $z_r\underset{r\to\infty}{\longrightarrow} z:=(\widetilde u,\overline u,\widehat W)$ a.s.
in $C([0,T]\times[0,1],\mathbb{R}^3)$, where the random field $z$ is
defined on $(\widehat\Omega,\widehat{\mathcal{F}},
(\widehat{\mathcal{F}_t})_{t\geq1},\widehat{\mathbb{P}})$, $\widehat
W$ is a Brownian sheet defined on this basis, and
$\widehat{\mathcal{F}_t}=\sigma(z(s,x), \, (s,x)\in [0,t]\times
[0,1])$ (and conveniently completed).
\item For every $r\geq 1$, the finite dimensional distributions of $z_r$ coincide with those
of the random field $\zeta_r:=\bigl(v^1_{j_r},v^2_{\ell_r},W\bigr)$,
and thus $\text{law}(z_r)=\text{law}(\zeta_r)$ for all $r\geq 1$.
\end{enumerate}

\smallskip

Note that $\widehat W_r$ is a Brownian sheet defined on
$(\widehat\Omega,\widehat{\mathcal{F}},
(\widehat{\mathcal{F}^r_t})_{t\geq1},\widehat{\mathbb{P}})$, where
$\widehat{\mathcal{F}^r_t}=\sigma(z_r(s,x), \\ (s,x)\in [0,t]\times
[0,1])$ (and conveniently completed).

\smallskip

We now fix $(t,x)\in[0,T]\times[0,1]$. Since the laws of $z_r$ and
$\zeta_r$ coincide and the first two components of $\zeta_r$ satisfy
the weak form \eqref{eq:6}, so do the components of $z_r$.
Namely,
 for all $\Phi\in C^\infty(\mathbb{R}^2)$ with
$\Phi(t,0)=\Phi(t,1)=0$ for all $t$, it holds
\begin{align}
 \int_0^1 \widetilde{u}_r(t,\kappa_M(y))\Phi(t,y)\, \dd y = & \int_0^1 u_0(\kappa_M(y))\Phi(t,y)\, \dd y \nonumber \\
 & \quad + \int_0^t\int_0^1 \widetilde{u}_r(s,\kappa_M(y)) \left(\Delta_M \Phi(s,y) +
 \frac{\partial \Phi}{\partial s}(s,y)\right)\, \dd y\, \dd s \nonumber \\
 & \quad + \int_0^t\int_0^1 f(\widetilde{u}_r(\kappa_N^T(s),\kappa_M(y))) \Phi(s,y)\, \dd y\, \dd s \nonumber \\
 & \quad + \int_0^t\int_0^1 \sigma(\widetilde{u}_r(\kappa_N^T(s),\kappa_M(y))) \Phi(s,y)\, W(\dd s,\dd y),
 \quad \widehat{\mathbb{P}}\text{-a.s.},
 \label{eq:11}
\end{align}
for all $t\in [0,T]$, and also
\begin{align}
 \int_0^1 \overline u_r(t,\kappa_M(y))\Phi(t,y)\, \dd y = & \int_0^1 u_0(\kappa_M(y))\Phi(t,y)\, \dd y \nonumber \\
 & \quad + \int_0^t\int_0^1 \overline u_r(s,\kappa_M(y)) \left(\Delta_M \Phi(s,y) +
 \frac{\partial \Phi}{\partial s}(s,y)\right)\, \dd y\, \dd s \nonumber \\
 & \quad + \int_0^t\int_0^1 f(\overline u_r(\kappa_N^T(s),\kappa_M(y))) \Phi(s,y)\, \dd y\, \dd s \nonumber \\
 & \quad + \int_0^t\int_0^1 \sigma(\overline u_r(\kappa_N^T(s),\kappa_M(y))) \Phi(s,y)\, W(\dd s,\dd y),
 \quad \widehat{\mathbb{P}}\text{-a.s.},
 \label{eq:12}
\end{align}
for all $t\in [0,T]$. We recall that $\Delta_M$ denotes the discrete
Laplacian, which is defined by
\[
 \Delta_M\Phi(s,y):= (\Delta x)^{-2} \left\{\Phi(s,y+\Delta x) - 2\Phi(s,y) + \Phi(s,y-\Delta
 x)\right\},
\]
where we remind that $\Delta x=\frac1M$.

Taking $r\rightarrow\infty$ in the above formulas \eqref{eq:11} and
\eqref{eq:12}, and using Lemma \ref{limit}, we show that the random
fields $\tilde u$ and $\bar u$ are solutions of \eqref{eq:13}, and
hence of equation \eqref{heateq}, on the same stochastic basis
$(\widehat\Omega,\widehat{\mathcal{F}},(\widehat{\mathcal{F}_t})_{t\geq1},\widehat{\mathbb{P}})$.
Thus, by the pathwise uniqueness assumption, we obtain that
$\widetilde u(t,x)=\overline u(t,x)$ for all
$(t,x)\in[0,T]\times[0,1]\:\:$ $\widehat{\mathbb{P}}$-a.s. Hence, by
Lemma \ref{polish}, we get that $\{u^{M_{k},N_{k}}\}_{k\geq1}$
converges in probability to $u$, uniformly on $[0,T]\times[0,1]$,
the solution of the stochastic heat equation \eqref{heateq}.
\end{proof}

\section{Acknowledgement}
L. Quer-Sardanyons' research is supported by grants 2014SGR422 and MTM2015-67802-P.
This work was partially supported by the Swedish Research Council (VR) (project nr. $2013-4562$). 
The computations were performed on resources provided 
by the Swedish National Infrastructure for Computing (SNIC) at HPC2N, Ume{\aa} University. 

\bibliographystyle{plain}
\bibliography{biblio}

\end{document}